\documentclass[12pt,reqno]{amsart}

\usepackage{color}
\usepackage{amsmath, amsthm, amssymb}
\usepackage{amsfonts}
\usepackage[ansinew]{inputenc}
\usepackage[dvips]{epsfig}
\usepackage{graphicx}
\usepackage[english]{babel}
\usepackage{hyperref}

\usepackage{cite}
\usepackage{graphicx}
\usepackage{amscd}
\usepackage{color}
\usepackage{bm}
\usepackage{enumerate}

\usepackage{verbatim}
\usepackage{hyperref}
\usepackage{amstext}
\usepackage{latexsym}
\theoremstyle{plain}
\newtheorem{Theorem}{Theorem}[section]

\newtheorem{Proposition}[Theorem]{Proposition}
\newtheorem{Lemma}[Theorem]{Lemma}
\newtheorem{Remark}[Theorem]{Remark}

\newtheorem{Conjecture}[Theorem]{Conjecture}

\numberwithin{Theorem}{section}
\numberwithin{equation}{section}

\def\square{\vbox{
\hrule height .4pt \hbox{\vrule width .4pt height 7pt \kern 7pt
\vrule width .4pt} \hrule height .4pt }}
\def\QED{\hfill {$\square$}\goodbreak \medskip}

\newcommand{\average}{{\mathchoice {\kern1ex\vcenter{\hrule height.4pt
width 6pt depth0pt} \kern-9.7pt} {\kern1ex\vcenter{\hrule
height.4pt width 4.3pt depth0pt} \kern-7pt} {} {} }}

\def\R{\mathbb{R}}

\renewcommand{\a }{\alpha }

\newcommand{\D }{\Delta }

\newcommand{\e }{\varepsilon }

\renewcommand{\l }{\lambda }

\newcommand{\n }{\nabla }
\newcommand{\vp }{\varphi }

\newcommand{\s }{\sigma }

\renewcommand{\t }{t} 

\renewcommand{\o }{\omega }
\renewcommand{\O }{\Omega }

\newcommand{\ov}{\overline}

\newcommand{\be}{\begin{equation}}
\newcommand{\ee}{\end{equation}}

\newcommand{\de}{\partial}

\newcommand{\ti}{\widetilde}

\newcommand{\calU}{{\mathcal U}}

\renewcommand{\textbf}[1]{\begingroup\bfseries\mathversion{bold}#1\endgroup}

\newcommand{\N}{\mathbb{N}}

\newcommand{\Z}{\mathbb{Z}}

\newcommand{\cH}{{\mathcal H}}

\newcommand{\cL}{{\mathcal L}}

\newcommand{\cN}{{\mathcal N}}

\newcommand{\cP}{{\mathcal P}}

\newcommand{\cU}{{\mathcal U}}

\newcommand{\cX}{{\mathcal X}}
\newcommand{\cY}{{\mathcal Y}}

\newcommand{\eps}{\varepsilon}

\DeclareMathOperator{\id}{id}

\renewcommand{\epsilon}{\varepsilon}

\begin{document}

\title[The Schiffer problem on the cylinder and on the $2$-sphere] 
{The Schiffer problem on the cylinder and on the $2$-sphere}

\author{Mouhamed Moustapha Fall}
\address{M. M. F.: African Institute for Mathematical Sciences in Senegal, KM 2, Route de
Joal, B.P. 14 18. Mbour, Senegal.}
\email{mouhamed.m.fall@aims-senegal.org}


\author[I. A. Minlend]{Ignace Aristide Minlend} 
\address{Faculty of Economics and Applied Management, University of Douala,  BP 2701, Douala, Littoral Province, Cameroon}
\email{\small{ignace.minlend@univ-douala.com} }

\author{Tobias Weth}
\address{T.W.:  Goethe-Universit\"{a}t Frankfurt, Institut f\"{u}r Mathematik.
Robert-Mayer-Str. 10 D-60629 Frankfurt, Germany.}

\email{weth@math.uni-frankfurt.de}


\begin{abstract}
  We prove the existence of a family of compact  subdomains  $\O$ of the flat cylinder
  $\R^N\times \R/2\pi\Z$ for which the Neumann eigenvalue problem for the Laplacian on $\Omega$ admits eigenfunctions with constant Dirichlet values on $\partial \Omega$. These domains $\Omega$ have the property that their boundaries $\partial \Omega$ have nonconstant principal curvatures. In the context of ambient Riemannian manifolds, our construction provides the first examples of such domains whose boundaries are neither homogeneous nor isoparametric hypersurfaces. The functional analytic approach we develop in this paper overcomes an inherent loss of regularity of the problem in standard function spaces. With the help of this approach, we also construct a related family of subdomains of the $2$-sphere $S^2$. By this we disprove a conjecture in \cite{Souam}.
\end{abstract}
\maketitle


\maketitle

\section{Introduction and main result}

A long standing open problem attributed to Schiffer, see e.g. S-T Yau \cite[Problem 80, p. 688]{Yau},  is to decide for which smooth bounded subdomains $\O \subset \R^N$ there exist a constant   $\mu >0$ and a solution  $u\ne 0$ to the overdetermined Neumann  problem
\begin{align*}
(\textrm{N}_\mu):  
  \left \{
    \begin{aligned}
        \D u +\mu  u &=  0 && \qquad \text{in $\Omega$,}\\
             |\n u| &=0 &&\qquad \text{on $\partial \Omega$,}\\
              u &=\text{const}  &&\qquad \text{on $\partial \Omega$.}
    \end{aligned}
       \right.
\end{align*}
It is conjectured that round balls are the only smooth bounded domains in $\R^N$ admitting a solution of $(\textrm{N}_{\mu})$ for some $\mu>0$. This conjecture is by now widely known as the \emph{Schiffer conjecture}. 

The conjecture is strongly connected to the so called Pompeiu  problem \cite{L.BrownSchreiberTaylor, Berenstein, S.A. Williams, Zalcman}.  A bounded domain $\O \subset \R^{N}$ is said to have  the \emph{Pompeiu property}  if $f \equiv 0$ is the only continuous function on $\R^{N}$ for which
$$
\int_{\sigma(\O)}f \,dx = 0 \qquad \text{for every rigid motion $\sigma$ of $\R^N$.}
$$
The Pompeiu  problem  consists  in characterizing the class of domains in $\R^{N}$  having  the Pompeiu  property. In  1976,  Williams  \cite{S.A. Williams}  proved that a smooth domain $\O \subset \R^N$, with $\partial \Omega$ homemorphic to a sphere, fails to satisfy the Pompeiu property if and only if  problem  $(\textrm{N}_{\mu})$ admits a solution with $\mu>0$ (see also \cite{L.BrownSchreiberTaylor,Berenstein-S}).
Williams also stated a weaker version of the Schiffer conjecture for the subclass of bounded domains which are homeomorphic to a ball.

Up to date, only partial results are available on the Schiffer conjecture.  In  \cite{Berenstein, Berenstein Yang}, Berenstein and Yang proved that the existence of infinitely many eigenvalues to $(\textrm{N}_\mu)$ implies that $\Omega$ must be a round ball. Moreover, letting
$$
0= \nu_1(\Omega) \le  \nu_2(\Omega) \leq  \dots
$$
denotes the sequence of Neumann eigenvalues of $-\Delta$ on $\Omega$, it has been proved by Deng in \cite{Deng} that a smooth bounded domain $\Omega \subset \R^2$ with $\partial \Omega$ connected is a ball if
\begin{itemize}
\item[(i)] $(\textrm{N}_{\mu})$ admits a solution for some $0<\mu<\nu_8(\Omega)$ or 
\item[(ii)] $\Omega$ is strictly convex, centrally symmetric  and $(\textrm{N}_{\mu})$ admits a solution for some $0<\mu<\nu_{13}(\Omega)$.
\end{itemize}
See also \cite{P. Aviles,WillmsGladwell,Friedlander} and the references in \cite{Deng} for earlier results in this direction. 

In \cite{V.Shklover}, Shklover studied a  generalization of  Schiffer's problem to an ambient Riemannian manifold $(M,g)$.
In order to discuss this problem in detail, it is convenient to call, here and in the following, a smooth bounded domain $\Omega \subset M$ a {\em Schiffer domain} if problem $(\textrm{N}_{\mu})$ has a solution in $\Omega$ for some $\mu>0$. Here $-\Delta=-\Delta_g$ has to be understood as the Laplace-Beltrami operator on $M$.

In \cite[P. 540]{V.Shklover}, Shklover discussed the question whether all boundary components of Schiffer domains $\Omega \subset M$ must be homogeneous hypersurfaces in $M$, which means that they are orbits of subgroups of the group of isometries of $M$. He then answered this question negatively by providing specific examples of Schiffer domains in manifolds $M$ of constant sectional curvature. These Schiffer domains are bounded by hypersurfaces which are isoparametric but not homogeneous (see e.g. \cite[Definition 3]{V.Shklover} for a definition of isoparametric hypersurfaces). By Cartan's theorem \cite{cartan}, the isoparametricity property of these hypersurfaces implies that their principal curvatures are constant. The construction in \cite{V.Shklover} yields, in particular, interesting examples in the case of the unit sphere  $M=\mathbb{S}^{N}$, $ N\geq 3$, which is addressed in detail in \cite[Section 2]{V.Shklover}. Later, Souam \cite{Souam} studied the $2$-sphere $M=S^2$ and proved the following classification result: If $\O \subset S^2$ is a $C^{2,\alpha}$-domain admitting a solution of $(\textrm{N}_{\mu})$ and either
\begin{itemize}
\item[(i)] $\mu=2$ and $\Omega$ is simply connected
\item[(ii)] or $\mu= \lambda_2$, the second Dirichlet eigenvalue of $-\D$ on $\Omega$,   
\end{itemize}
then $\Omega$ is a geodesic disc. He then also formulated the following

\begin{Conjecture} (see \cite[Conjecture 1.1]{Souam})\\
\label{conjecture-souam}  
If $\Omega \subset S^2$ is a sufficiently regular Schiffer domain, then $\Omega$ is either a geodesic
disk or a round symmetric annulus.
\end{Conjecture}

The purpose of the present paper is twofold. In the first part of the paper, we provide the first example of a Schiffer domain in a Riemannian manifold $(M,g)$ of constant sectional curvature whose (connected) boundary has nonconstant principal curvatures, therefore it is neither homogeneous nor isoparametric (again by Cartan's theorem \cite{cartan}). Here we consider the ambient manifold $M$ given as the flat cylinder $\R^N\times \R/2\pi\Z$ endowed with the flat metric. In this case, Schiffer domains correspond to $2\pi$ periodic (in the last component) subdomains $\Omega \subset \R^N \times \R$ for which $(\textrm{N}_{\mu})$ admits a solution for some $\mu>0$. The domains we construct in the present paper will be cylindrical subgraphs close to the straight cylinder $B_1 \times \R$, where, here and in the following, $B_1$ denotes the unit ball in $\R^N$. In the final section of the paper, we use the methods devoloped for the flat cylinder to construct regular Schiffer domains $\Omega \subset S^2$. These domains are given as open neighborhoods of the equator $S^1 \times \{0\}$ in $S^2$ bounded by pairs of curves $x \mapsto  \left(\cos x \cos h(x),\sin x \cos h(x), \sin \bigl( \pm h(x)\bigr) \right)$ with associated nonconstant $2\pi$-periodic functions $h$. Consequently, the existence of these domains disproves Conjecture~\ref{conjecture-souam}.  

To state our first main result, we fix $\alpha \in (0,1)$ and define  by  $C^{2,\alpha}_{p}(\R)$  the space of  $2\pi$ periodic and even $C^{2,\alpha}$-functions on $\R$, and we let $\cP^{2,\alpha}_{p}(\R)$ denote the open subset of strictly positive functions in $C^{2,\alpha}_{p}(\R)$. For a function $h \in \cP^{2,\alpha}_{p}(\R)$, we  define  the domain               
\begin{equation}\label{eq:PertTorus}
\Omega_h:= \left\{\left(t,x \right)\in  \R^N\times \R \::\: |t|<\frac{1}{h(x)}  \right\}.
\end{equation}

\bigskip
Our main result  is the following.  
\begin{Theorem}\label{Theo1-ND}
Let $N,m  \in \N$ be positive integers.  Then there exist  $\e>0$ and (explicit) constants $\mu_0, \kappa >0$, $\beta,\gamma \in \R \setminus \{0\}$, depending only on $N$ and $m$, and a smooth curve
$$
(-{\e},{\e}) \to   (0,+\infty) \times  \cP^{2,\alpha}_{p}(\R) ,\qquad s \mapsto (\mu_s,h_s)
$$
with $\mu_s \big|_{s=0}= \mu_0$, 
$$
h_s(x)= \kappa \sqrt{\mu_s} + s \beta \cos (x) + o(s) \qquad \text{as $s \to 0$ uniformly on $\R$}
$$
and the property that the overdetermined boundary value problem
\begin{equation}\label{eq:solved-main-ND}
  \left \{
    \begin{aligned}
       \D w_s+ {\mu_s} w_s &=  0 && \qquad \text{in $ \Omega_{ h_s}$,}\\
             w_s&=1 &&\qquad \text{on $\partial \Omega_{ h_s}$,}\\
              |\n w_s | &=0 &&\qquad \text{on $\partial  \Omega_{ h_s}$}
    \end{aligned}
       \right.
\end{equation}
admits a classical solution $w_s$ for every $s \in (-\eps,\eps)$ which is radial in $t$ and
$2\pi$-periodic and even in $x$. Moreover, we have
\begin{equation}
  \label{eq:w-s-expansion}
w_s(\frac{t}{h_s(x)},x)= U_m(|t|) +s\bigl\{\phi_1(|t|)+  \gamma \,|t| U_m'(|t|)\bigr\} \cos (x) + o(s)  \quad \text{as $s \to 0$}
\end{equation}
uniformly on $B_1 \times \R$, where $t \mapsto U_m(|t|)$ is the $m$-th nonconstant radial Neumann eigenfunction of the Laplacian on the unit ball $B_1$ of $\R^N$, normalized such that $U_m(1)=1$, and $t \mapsto \phi_1(|t|)$ is a (suitably normalized) positive first Dirichlet eigenfunction of the Laplacian on the unit ball $B_1$.
\end{Theorem}

Some remarks are in order. 
\begin{Remark}\label{eq:RKmain-ND}
  {\rm
(i) As observed by  Kinderlehrer  and  Nirenberg \cite{KinderlehrerNirenberg}, the domains in  Theorem \ref{Theo1-ND} are   analytic.   It is clear from (\ref{eq:w-s-expansion}) that the solution  $w_s$  changes sign on  $ \Omega_{ h_s}$. More precisely, it can be shown that $w_s$ has precisely $m$ nodal domains in $\Omega_{h_s}$ for $s$ small,
since $U_m$ has $m-1$ nondegenerate zeros in $(0,1)$ and the expansion (\ref{eq:w-s-expansion}) even holds in $C^1$-sense, see Section~\ref{sec:proof-theor-refth} below for details.\\
     (ii) To write the constants $\mu_0, \kappa,\beta,\gamma$ and the functions $U_m, \phi_1$ in Theorem~\ref{Theo1-ND} explicitly, we need to fix some notation regarding Bessel functions. Let $J_{\nu}$ denote the Bessel function of the first kind of order $\nu>-1$, and let
$$
0< j_{\nu,1} < j_{\nu,2} < j_{\nu,3} < \dots
$$
denote the ordered sequence of zeros of $J_\nu$. We also put
$$
I_\nu(r):= r^{-\nu} J_\nu(r) \qquad \text{for $\nu>-1, r>0$.}
$$
Then the function $U_m$ in Theorem~\ref{Theo1-ND} is given by 
\begin{equation*}
r \mapsto U_m(r)=\dfrac{I_{N/2-1}(j_{N/2,m}\, r)}{I_{N/2-1}(j_{N/2,m})}
\end{equation*}
and the function $\phi_1$ is given by $r \mapsto \phi_1(r)= I_{N/2-1}(j_{N/2-1,1}\, r)$.
Moreover, we have $\mu_0 = \frac{j_{N/2,m}^2}{j_{N/2,m}^2-j_{N/2-1,1}^2}$, $\kappa = \frac{1}{j_{N/2,m}}$, $\beta = -\frac{j_{N/2-1,1}^2 I_{N/2-1}(j_{N/2,m}) I_{N/2}(j_{N/2-1,1}) }{j_{N/2,m}^2 I_{N/2-1}''(j_{N/2,m})\sqrt{j_{N/2,m}^2-j_{N/2-1,1}^2 }}$ and
$\gamma= -\frac{j_{N/2-1,1}^2 I_{N/2-1}(j_{N/2,m})  I_{N/2}(j_{N/2-1,1}) }{j_{N/2,m}^2 I_{N/2-1}''(j_{N/2,m})}.
$
For details, see the proof of Theorem~\ref{Theo1-ND} in Section~\ref{sec:proof-theor-refth} below.
}
\end{Remark}\vspace{0.5ex}

\begin{Remark}\label{eq:RKmain-ND-1}
{\rm   
The case $N=1$ in Theorem~\ref{Theo1-ND} is of particular interest. In this case, the $m$-th nonconstant radial (i.e., even) Neumann eigenfunction $t \mapsto U_m(|t|)$ of the Laplacian on the unit ball $B_1=(-1,1) \subset \R$ with $U_m(1)=1$ is given by $t \mapsto U_m(|t|)= (-1)^m \cos \bigl( m \pi |t|\bigr)$. One might ask if in this case a similar bifurcation result can be proved with $U_m$ replaced by an odd eigenfunction of the type $t \mapsto \sin \bigl((m-\frac{1}{2})\pi t\bigr)$. In this case, the corresponding overdetermined problem needs to be 
\begin{equation}
  \label{eq:overdetermined-squared}
  \left\{
    \begin{aligned}        \D u +\mu  u &=  0 && \qquad \text{in $\Omega$,}\\
      u^2 &=const&&\qquad \text{on $\partial \Omega$,}\\
      |\n u| &=0 &&\qquad \text{on $\partial \Omega$,}        
    \end{aligned}
       \right.
     \end{equation}
taking into account that the straight cylinder has two boundary components in the case $N=1$ and odd functions in the $t$-variable are only equal up to sign on these boundary components. In fact, it follows from the results of the first and the third author in \cite{Fall-Weth} that such a bifurcation does not occur in the case $m=1$. More precisely, in \cite{Fall-Weth} problem~(\ref{eq:overdetermined-squared}) is considered for domains $\Omega \subset \R \times \cN$ with a closed Riemannian manifold $\cN$ and with $\mu=\mu_2(\Omega)$ (the first nontrivial Neumann eigenvalue of $-\Delta$ on $\Omega$). Moreover, it is shown that in this case (\ref{eq:overdetermined-squared}) admits a solution if and only if $\Omega = (-r,r) \times \cN$ for some $r \ge r_\cN>0$, and then the solution is given as a scalar multiple of the function $(t,x) \mapsto \sin(\frac{\pi t}{2r})$. This rigidity result applies in particular in the case $\cN = S^1= \R/2\pi\Z$ and excludes an analogue of Theorem~\ref{Theo1-ND} where the function $t \mapsto \sin \bigl(\frac{\pi}{2} t\bigr)$ would appear in place of $t \mapsto U_m(|t|)$ in the expansion~(\ref{eq:w-s-expansion}).}
\end{Remark}\vspace{0.5ex}

The main strategy of the proof of Theorem~\ref{Theo1-ND} is to transform our problem to a parameter dependent nonlinear operator equation in Banach spaces which can be solved by means of the Crandall-Rabinowitz bifurcation theorem \cite{M.CR}. We now comment on the main steps and difficulties of the proof and outline the organization of this paper.

In Section  \ref{eq.setting}, we first rephrase  the main problem  \eqref{eq:perturbed-strip-ND} to an equivalent problem  \eqref{eq:Proe1-ss2-ND} on the fixed domain $\O_*:=  B_1\times \R$.  This leads  us  to considering an equation of the function $F_\lambda(u_m + u,1+h)=0$ with unknown functions $u \in C^{2,\alpha}_{p,rad}(\overline{\O_*})$ and $h \in C^{2,\alpha}_p(\R)$ for some $\alpha \in (0,1)$. Here $u_m$ is related to $U_m$ in Theorem~\ref{Theo1-ND} by $u_m(t,x) = U_m(|t|)$, and $C^{2,\alpha}_{p,rad}(\overline{\O_*})$ denotes the space of $C^{2,\alpha}$-functions $u=u(t,x)$ which are radial in $t$ and $2\pi$ periodic and even in $x$.  However, the linearization of $F_\lambda$ at $(0,0)$ turns out to be not of Fredholm type between classical Hölder spaces and comes with an apparent loss of derivatives.  With the help of a suitable transformation, we overcome this problem by eliminating the variable $h$. This allows us to reduce the problem to an equation of the type $G_\lambda(u)=0$ for some function $(\lambda,u) \mapsto G_\lambda(u)$. However, in this elimination procedure, we are lead to substitute a function $h = h_u$ depending on derivatives of $u$, see Remark~\ref{rem:const-sol} below. Therefore we have to leave the framework of standard Hölder spaces and have to consider both $F$ and $G_\lambda$ as maps between (open subsets of) new tailor made Banach spaces $X_2^D$ and $Y$, see Section~\ref{sec:functional-setting} below. We believe that this strategy might have further applications in this context to overcome an apparent loss of derivatives.
 
In Section~\ref{sec:functional-setting}, we then compute the linearised operator  $D_uG_\lambda(0): X_2^D \to Y$ and show that it is a Fredholm operator of index zero.

Then, in Section~\ref{eq:kernel}, we determine a parameter value  $\l_{m}$  for which  $D_uG_{\l_{m}}(0):X_2^D\to Y$  has a one  dimensional kernel and the transversality condition  in the Crandall-Rabinowitz bifurcation theorem \cite{M.CR} holds.

In Section~\ref{eq:ProofTheo1-ND}, we then complete the proof of Theorem~\ref{Theo1-ND}.

Finally, in Section \ref{sec:schiffer-problem-s2}, we study problem $(\textrm{N}_{\mu})$ on the $2$-sphere $S^2$ and prove Theorem~\ref{Theo1-ND-S2} below, which provides the existence of Schiffer domains given as open neighborhoods of the equator $S^1 \times \{0\}$ in $S^2$. While the functional analytic framework for this problem is very similar to the one developed for the flat cylinder, the analysis differs considerably, see Remark~\ref{eq:RK-sphere} below.\\

To state our main result on the round $2$-sphere $S^2$, we fix $\alpha \in (0,1)$ and define 
\begin{equation}
  \label{eq:def-p-s-2}
  \cP^{2,\alpha}_{S^2}(\R):= \bigl\{h \in C^{2,\alpha}_{p}(\R)\::\: 0< h < \frac{\pi}{2} \quad \text{on $\R$} \bigr\}. 
\end{equation}
and, for $h \in \cP^{2,\alpha}_{S^2}(\R)$, the subdomain
\begin{equation}\label{eq:Pertsphere-1}
\widetilde \Omega_h:= \left\{\left(\cos x \cos \bigl(h(x)t\bigr),\sin x \cos \bigl(h(x)t\bigr), \sin \bigl(h(x)t\bigr) \right)\in S^2 \::\: |t|<1, \, x \in \R  \right\}
\end{equation}
of $S^2$ which is an open neighborhood the equator $S^1 \times \{0\} \subset S^2$ bounded by the two curves $x \mapsto  \left(\cos x \cos h(x),\sin x \cos h(x), \sin \bigl( \pm h(x)\bigr) \right)$.

\begin{Theorem}\label{Theo1-ND-S2}
For every $\lambda_0 \in (0,1)$ there exists $\ell_0= \ell_0(\lambda_0) \in \N$ with the property that for every $\ell \in \N$ with $\ell \ge \ell_0$
there exists $\e>0$ and a smooth curve
$$
(-{\e},{\e}) \to   (0,+\infty) \times (0,\lambda_0) \times \cP^{2,\alpha}_{S^2}(\R) ,\qquad s \mapsto (\mu_s,\xi_s, h_s)
$$
with 
\begin{equation}
  \label{eq:h-s-expansion}
h_s(x)= \xi_s + s \cos (\ell x) + o(s) \qquad \text{as $s \to 0$ uniformly on $\R$}
\end{equation}
and the property that the overdetermined boundary value problem
\begin{equation}\label{h-Neu-over-S2}
  \left \{
    \begin{aligned}
\D w+ {\mu} w &=  0 && \qquad \text{in $  \widetilde\Omega_{h}$,}\\
             w&=1 &&\qquad \text{on $\partial  \widetilde \Omega_{h}$,}\\
              |\n w|  &=0 &&\qquad \text{on $  \partial \widetilde \Omega_{h}$}
    \end{aligned}
       \right.
\end{equation}
with $\mu = \mu_s$, $h=h_s$ admits, for every $s \in (-\eps,\eps)$, a classical solution. Here, $\Delta$ denotes the Laplace-Beltrami operator on $S^2$ with respect to the standard round metric.
\end{Theorem}

\begin{Remark}\label{eq:RK-sphere}
{\rm   (i) As noted before, Theorem~\ref{Theo1-ND-S2} disproves Conjecture~\ref{conjecture-souam} since, by (\ref{eq:h-s-expansion}), the functions $h_s$ are nonconstant for $s$ nonzero but close to zero.\\
  (ii) While the functional analytic framework of the proof of Theorem~\ref{Theo1-ND-S2} is very similar to the one of the proof of Theorem~\ref{Theo1-ND}, there are some differences which require additional work. First, we cannot rescale problem (\ref{h-Neu-over-S2}) in the same way as we do in the beginning of Section~\ref{sec:preliminaries} for the case of the ambient flat cylinder, where, after rescaling, the bifurcation parameter appears simply as a factor of the $\partial_{xx}$-term. Moreover, while there still exists a smooth family of one-dimensional solutions of (\ref{h-Neu-over-S2}) in the case $h \equiv \lambda \in (0,\frac{\pi}{2})$ depending only on the distance to the equator $S^1 \times \{0\} \subset S^2$, this family and the associated family of Dirichlet eigenfunctions are more difficult to analyze than the family of Bessel functions.} 
\end{Remark}\vspace{0.5ex}

We close this introduction by discussing some related work on overdetermined Dirichlet problems. In particular, we mention closely related work by Sicbaldi \cite{P.Sicbaldi} as well as Sicbaldi and Schlenk \cite{Sicbaldi-Schlenk} on the overdetermined Dirichlet eigenvalue problem 
\begin{align*}
(\textrm{D}_{\mu}):   
  \left \{
    \begin{aligned}
        \D u +\mu  u &=  0 && \qquad \text{in $\Omega$,}\\
              u &=0 &&\qquad \text{on $\partial \Omega$}\\
               |\n u| &=b\ne 0  &&\qquad \text{on $\partial \Omega$,}
    \end{aligned}
       \right.
\end{align*}
In these works, the authors also used the Crandall-Rabinowitz bifurcation theorem to construct periodic domains bifurcating from straight cylinders which admit solutions of $(\textrm{D}_{\mu})$ for $\mu = \lambda_1(\Omega)$ where $\lambda_1(\Omega)$ is the first Dirichlet eigenvalue of $\Omega$, considered as a subdomain of $\R^N\times \R/T\Z$. In the case $N \ge 3$, these domains provide counterexamples to a conjecture of Berestycki, Caffarelli and Nirenberg  \cite[p. 1110]{BeCahNi}  claiming,  if  $\O$  is a  smooth bounded  domain such that such that $\R^{N+1}\setminus \ov \O$ is connected and there exists a solution  to  
\begin{align}\label{eq:BCN}   
  \left \{
    \begin{aligned}
        \D u + f(u) &=  0 && \qquad \text{in $\Omega$,}\\
              u &=0 &&\qquad \text{on $\partial \Omega$}\\
               |\n u| &=const.  &&\qquad \text{on $\partial \Omega$,}
    \end{aligned}
       \right.
\end{align}
for some Lipschitz function $f$, then $\O$  should be a half-space, a ball, the complement of a ball, or a circular-cylinder-type domain $\R^j\times C$ (up to rotation and translation), where $C$ is a ball or a complement of a ball in $\R^{N-j}.$  For  further results regarding \eqref{eq:BCN} including general ambient  spaces,  we refer the reader to the references \cite{Alessandrini, Gazzola, Greco,   Morabito Sicbaldi, Philippin,  GarofaloLewis, Fall-MinlendI-Weth, PaynePhilippin, Lamboley, FragalaGazzola, BrockHenrot, Reichel, PhilippinPayne, BerchioGazzolaWeth, Ku-Pra, FraGazzolaKawohl, farina-valdinoci,farina-valdinoci:2010-1,farina-valdinoci:2010-2,farina-valdinoci:2013-1, farina-valdinoci:2013-2, Fall-MinlendI-Weth2, Fall-Minlend}.

We point out two key differences between the existing works on $(\textrm{D}_{\mu})$, (\ref{eq:BCN}) and those on  $(\textrm{N}_{\mu})$. First, there is no loss of derivatives present in the constructions related to $(\textrm{D}_{\mu})$ and (\ref{eq:BCN}). In contrast, we have to deal with it in the present paper, and it also appears in other recent works on overdetermined problems with homogeneous Neumann boundary conditions, see e.g. \cite{cui-arxiv} and the references therein. Second, while $(\textrm{N}_{\mu})$ does not possess non-constant positive solutions, most of the work on $(\textrm{D}_{\mu})$ and (\ref{eq:BCN}) deals with positive solutions $u$ and uses related uniqueness and nondegeneracy properties. We are only aware of the following two very recent works on the overdetermined Dirichlet problem for sign changing solutions. In  \cite{I.A.Minlend2}, the second author addressed the existence of sign-changing solution to a variant of the Dirichlet eigenvalue  problem  $(\textrm{D}_\mu)$ in unbounded periodic domains. Moreover,
in the very interesting recent preprint \cite{Ruiz-arxiv}, the nonlinear problem (\ref{eq:BCN}) is considered with a specific nonlinearity which gives rise to the existence of sign changing solutions in perturbations of the unit ball.\\

\noindent \textbf{Acknowledgements}: 
  I.A. Minlend is supported by the Alexander von Humboldt foundation.
This work was  carried  out when  I.A. Minlend was visiting the Goethe University Frankfurt am Main.  He is grateful to the Institute of Mathematics for its hospitality.

\section{The pull back problem}\label{eq.setting} 

\subsection{Preliminaries}
\label{sec:preliminaries}
Recall that we are looking for $\mu>0$ and a nonconstant function $h \in \cP^{2,\alpha}_{p}(\R)$ with the property that the overdetermined problem
\begin{equation}\label{h-Neu-over}
  \left \{
    \begin{aligned}
       \D w+ {\mu} w &=  0 && \qquad \text{in $ \Omega_{h}$,}\\
             w&=1 &&\qquad \text{on $\partial \Omega_{h}$,}\\
              |\n w | &=0 &&\qquad \text{on $\partial  \Omega_{h}$}
    \end{aligned}
       \right.
\end{equation}
admits a solution. This problem is equivalent to finding $\mu>0$ and nonconstant $h \in \cP^{2,\alpha}_{p}(\R)$ with the property that 
\begin{equation}
  \label{eq:perturbed-strip-ND-0}
  \left \{
    \begin{aligned}
  L_{\lambda,\mu} u &=  0 && \qquad \text{in $\Omega_h$,}\\
           u &=1 &&\qquad \text{on $\partial \Omega_h$,}\\
             |\n u| &=0 &&\qquad \text{on $\partial \Omega_h$.}
    \end{aligned}
       \right.
     \end{equation}
     admits a solution for some $\lambda>0$, where the operator $L_{\lambda,\mu}$ is given by
     $$
     L_{\lambda,\mu}:= 
     \D_{\t} + \lambda \partial_{xx}  +  \mu \id.
     $$
     Indeed, a function $u \in C^2(\Omega_h)$ is a solution of  \eqref{eq:perturbed-strip-ND-0} if and only if the function 
\begin{equation}
  \label{eq:changsoluD}
w^{\lambda}\in C^2(\Omega_h),\qquad   w^{\lambda} (t, x):= u( t/\sqrt{\lambda}, x)
\end{equation}
solves (\ref{h-Neu-over}) with $h$ replaced by $\frac{h}{\sqrt{\lambda}}$ and $\mu$ replaced by $\frac{\mu}{\lambda}$.
We first consider the special case $h \equiv 1$ in \eqref{eq:perturbed-strip-ND-0}, i.e., the case of the straight cylinder $\Omega_h = \Omega_1 = B_1 \times \R$.

In this case, for every $\lambda>0$, a solution of \eqref{eq:perturbed-strip-ND-0}
is given by $u(t,x)= U(|t|)$ if $U \in C^2([0,1])$ solves the (overdetermined) ODE eigenvalue problem
\begin{equation}
  \label{eq:ODE-eigenvalue}
U'' + \frac{N-1}{r}U'  +\mu U = 0 \quad \text{in $(0,1)$,}\qquad U'(0)=U'(1)=0,\quad U(1)=1.
\end{equation}
Let, here and in the following, $J_{\nu}$ denote the Bessel function of the first kind of order $\nu>-1$, and let
$$
0< j_{\nu,1} < j_{\nu,2} < j_{\nu,3} < \dots
$$
denote the ordered sequence of zeros of $J_\nu$. Moreover, let
$$
I_\nu(r):= r^{-\nu} J_\nu(r) \qquad \text{for $\nu>-1, r>0$.}
$$
For the sake of brevity, we also set
$$
j_n:= j_{\frac{N}{2},n} \qquad \text{for $n \geq 1$.}
$$
It is well known that (\ref{eq:ODE-eigenvalue}) admits a (unique) nonconstant solution if and only if 
\begin{equation}
  \label{eq:bessel-zeros}
\mu= j_{n}^2  \qquad \text{for some $n \geq 1$,}
\end{equation}
and in this case the solution is given by
\begin{equation}
  \label{eq:def-tilde-un}
r \mapsto U_n(r)=\dfrac{I_{N/2-1}(j_nr)}{I_{N/2-1}(j_n)}.
\end{equation}
Note in particular that
$$
U_n'(1)=   \dfrac{j_n I_{N/2-1}'(j_n)}{I_{N/2-1}(j_n)}=  -  \frac{j_n^2 I_{N/2}(j_n)}{I_{N/2-1}(j_n)}=0 \quad \text{for $n\geq 1$}
$$
since
 \begin{equation}\label{eq:derrec}
  I_\nu'(r) =-r I_{\nu+1}(r) \qquad \text{for $\nu>-1, r>0$.} 
 \end{equation}
 We now fix $m  \geq 1$, put
 \begin{equation}
\label{eq:opLlam}
L_{\lambda}:= L_{\lambda,j_m^2}= \D_{\t} + \lambda \partial_{xx} +  j_m^2 \id     
 \end{equation}
 and consider the $\lambda$-dependent overdetermined problem  
\begin{equation}
  \label{eq:perturbed-strip-ND}
  \left \{
    \begin{aligned}
  L_{\lambda} u &=  0 && \qquad \text{in $\Omega_h$,}\\
           u &=1 &&\qquad \text{on $\partial \Omega_h$,}\\
             |\n u| &=0 &&\qquad \text{on $\partial \Omega_h$.}
    \end{aligned}
       \right.
     \end{equation}
 It is convenient to transform (\ref{eq:perturbed-strip-ND}) to an equivalent problem on the fixed domain
 $$
 \Omega_*:= \O_1=  B_1 \times \R.
 $$
 Observe that, for a function $h \in \cP^{2,\alpha}_{p}(\R)$, the domain $\Omega_h$ is parameterized by the mapping
 $$ \Psi_h: \Omega_*  \to  \Omega_h , \quad  ( t,x)  \mapsto (\tau, x)=( \frac{\t}{h(x)},x),$$
with inverse
$$
\Psi^{-1}_h:  \Omega_h   \to  \Omega_*, \quad  ( \tau ,x)  \mapsto   ( h(x) \tau ,x).
$$
Hence \eqref{eq:perturbed-strip-ND} is equivalent to 
\begin{equation}
  \label{eq:perturbed-strip-ND-fixed-domain}
  \left \{
    \begin{aligned}
  L_{\lambda}^h u &=  0 && \qquad \text{in $\Omega_*$,}\\
           u &=1 &&\qquad \text{on $\partial \Omega_*$,}\\
             |\n u| &=0 &&\qquad \text{on $\partial \Omega_*$,}
    \end{aligned}
       \right.
     \end{equation}
     where the operator
     \begin{equation}
     \label{eqreladiffopets-ND}
     L_{\lambda}^h: C^{2}(\ov{\Omega_*}) \to C^{0}(\ov{\Omega_*})\quad \text{is defined by}\quad 
     L_{\lambda}^h u= \bigl(L_\lambda (u \circ \Psi_h^{-1})\bigr)\circ \Psi_h.
     \end{equation}
     Indeed, $u \in C^{2}(\ov{\Omega_*})$ solves (\ref{eq:perturbed-strip-ND-fixed-domain}) if and only if $u \circ \Psi_h^{-1}$ solves (\ref{eq:perturbed-strip-ND}). To calculate an explicit expression for $L_{\lambda}^h$, we fix $u\in C^{2}(\ov{\Omega_*})$ and note that 
\begin{align}\label{eqreladiffopets-ND-1}
[L^h_\lambda u] (h(x)t,x) = [L_\lambda v_h](t,x) \quad \text{for $(t,x) \in \Omega_h$}
\end{align}
with
\begin{align}
  \label{eqreladiffopets-ND-2}
v_h \in C^2(\ov{\Omega_h}),\qquad v_h(t,x)=u(h(x)t,x).
\end{align}
A direct computation yields 
\begin{align*}
L_\lambda^h v_h(\t,x) =&j_m^2 u(h(x)t,x) + \lambda \partial_{xx}u(h(x)t,x) +   h^2(x)\D_\t u(h(x)t,x)   \\
                      &+ \lambda h'(x)^2  \nabla^2_t u(h(x) t,x)[\t,\t] + 2 \lambda h'(x) \nabla_t \partial_{x}u(h(x) t,x) \cdot t\\
  &+  \lambda h''(x) \nabla_t u(h(x) t,x) \cdot t \qquad  \text{for $(\t,x) \in \Omega_h$.}
\end{align*}
Replacing $t$ by $\frac{t}{h(x)}$ therefore gives   
\begin{align} 
L_\lambda^h u(\t,x) =&j_m^2 u(\t,x) + \lambda \partial_{xx}u(\t,x) +   h^2(x) \D_\t u(\t,x)+ \lambda     \frac{h'(x)^2}{h(x)^2} \nabla^2_t u(\t,x)[\t,\t]   \nonumber\\
&+ 2 \lambda  \frac{ h'(x) }{h(x)}   {D}_{t} \partial_{x}u(\t,x) +  \lambda \frac{ h''(x) }{h(x)} {D}_t u(\t,x) \qquad  \text{for $(\t,x) \in \Omega_*$.} \label{eq:reldiffope-ND}
\end{align}
Here $\nabla_t$ and $\Delta_t$ denote the gradient and Laplacian with respect to the variable $\t \in \R^N$, and we have set 
\begin{equation}
  \label{eq:def-dt-u}
[{D}_\t v](t,x)= \nabla_\t v(t,x) \cdot \t \qquad \text{for functions $v \in C^1(\ov{\Omega_*})$.}
\end{equation}
We also note that 
\begin{equation}
\label{nabla-delta-comp-1}  
[{D}_\t {D}_\t v](t,x)={D}_t v(t,x) 
+ \nabla^2_t v(t,x)[\t,\t] \qquad \text{for $v \in C^2(\ov{\Omega_*})$ and $(t,x) \in \Omega_*$.} 
\end{equation}
Hence we can rewrite \eqref{eq:reldiffope-ND} shortly as 
\begin{align} 
L_\lambda^h u &=j_m^2 u + \lambda \partial_{xx}u +   h^2\D_\t u+ \lambda     \frac{(h')^2}{h^2}{D}_\t {D}_\t u   \nonumber\\
&\qquad + 2 \lambda  \frac{ h'}{h}   {D}_{t} \partial_{x}u +  \lambda  \Bigl(\frac{ h''}{h}-\frac{h'^2}{h^2}\Bigr)  {D}_{t}u \qquad  \text{in $\Omega_*$.} \label{eq:reldiffope-ND-alt}
\end{align}
Here we identify the function $h \in \cP^{2,\alpha}_{p}(\R)$ with the function $(\t,x) \mapsto h(x)$ defined on $\ov{\Omega_*}$, and we do the same with $h'$ and $h''$.

Finally, it is easy to see that the  boundary conditions in (\ref{eq:perturbed-strip-ND-fixed-domain}) are equivalent to 
$$
u=1, \qquad  {D}_\t u=0   \qquad\textrm{ on $\de\O_*$}. 
$$
Hence (\ref{eq:perturbed-strip-ND-fixed-domain}) is equivalent to 
\begin{align}\label{eq:Proe1-ss2-ND}
 \begin{cases}
    L^h_\lambda u= 0 & \quad \textrm{ in}\quad  \Omega_* \\
u=1&   \quad \textrm{ on}\quad \partial \Omega_* \\
{D}_t u  = 0  &  \quad \textrm{ on}\quad \partial \Omega_*,
  \end{cases}
  \end{align}
where $L^h_\l$ is given by \eqref{eq:reldiffope-ND}.

\section{Functional setting}
\label{sec:functional-setting}
We    introduce  the spaces where we wish to solve our reduced problem.
For fixed $\a\in (0,1)$ and $k \in \N \cup \{0\}$, we set 
$$
C^{k,\a}_{p,rad}(\overline \Omega_*):= \{ u \in C^{k,\alpha}(\overline \Omega_*)\::\: \text{$u$ is radial in ${\t}$, $2\pi$ periodic and  even in $x$ } \},
$$
endowed with the norm $
u \mapsto \|u\|_{C^{k,\alpha}}:= \|u\|_{C^{k,\alpha}(\ov{\O_*})} .
$  We then define 
$$
X_k:= \{ u \in C^{k,\a}_{p,rad}(\overline \Omega_*)\::\:  {D}_\t u \in C^{k,\alpha}(\overline \Omega_*) \},
$$
endowed with the norm
$$
u \mapsto \|u\|_{k}:= \|u\|_{C^{k,\alpha}}+\|{D}_\t u \|_{C^{k,\alpha}}.
$$
\begin{Remark}
  \label{sec:functional-setting-2}
  {\rm (i) In the case $k=0$, the existence of the directional derivative ${D}_\t u$,
    defined e.g. by 
    $$
    {D}_\t u(t,x)= \frac{d}{d\s}\Bigl|_{\s=1} u(\sigma t,x) \qquad \text{for $(t,x) \in \Omega_*$,}
    $$
    and its $C^\alpha$-continuity up to the boundary is assumed by definition for $u \in X_0$.\\
 (ii) It is important to note that the differential operators $\partial_x$ and ${D}_t$ commute on $X_1$. More precisely, the existence and continuity of the derivative $\partial_x {D}_t u$ for $u \in X_1$ also implies the existence of $ {D}_t \partial_x u$, and the equality ${D}_t \partial_x u= \partial_x {D}_t u$ holds in $\ov{\Omega_*}$. For a proof of this fact (which is sometimes called Peano's Lemma), see e.g. \cite[pp. 235-236]{rudin}.
  Similar statements follow inductively for functions in $X_k$ and higher order derivatives.}
\end{Remark}\vspace{0.5ex}

Next we consider the closed subspaces
$$
X_k^D:=\{u \in X_k \::\: \text{$u= 0$ on $\partial \Omega_*$}\},
$$
and
$$
X_k^{DN}:=\{u \in X_k \::\: \text{$u={D}_\t u = 0$ on $\partial \Omega_*$}\},
$$
both also endowed with the norm $\|\cdot\|_k$. It is straightforward to see that these spaces are Banach spaces.

Finally, a key role in this paper will be played by the Banach space
$$
Y:= C^{1,\alpha}_{p,rad}(\overline \Omega_*) + X_0^D \; \subset \; C^{0,\alpha}_{p,rad}(\overline \Omega_*),
$$
which is endowed with the norm
$$
\|f\|_{Y}:= \inf \Bigl \{\|f_1\|_{C^{1,\alpha}} + \|f_2\|_{0} \::\: f_1 \in C^{1,\alpha}_{p,rad}(\overline \Omega_*), \; f_2 \in X_0^D,\; f= f_1 + f_2 \Bigr\}.
$$

\begin{Remark}
\label{Remark-Banach-space-quotient}  
{\rm The norm $\|\cdot \|_{Y}$ is the standard norm for the sum of two (embedded) Banach spaces. To see that it turns $Y$ into a Banach space, we note that $Y$ is isomorphic to the quotient space
$$
\bigl(C^{1,\alpha}_{p,rad}(\overline \Omega_*) \times  X_0^D\bigr) / V
$$
where $V \subset C^{1,\alpha}_{p,rad}(\overline \Omega_*) \times  X_0^D$ is the closed subspace of pairs $(f,f)$ with $f \in C^{1,\alpha}_{p,rad}(\overline \Omega_*)\cap  X_0^D$. An isomorphism is induced by the map
$$
C^{1,\alpha}_{p,rad}(\overline \Omega_*) \times  X_0^D  \to Y, \qquad (f_1,f_2) \mapsto f_1 - f_2
$$
whose kernel is precisely $V$. 
}\end{Remark}\vspace{0.5ex}

As in Section~\ref{eq.setting}, we now fix $m  \geq 1$ and consider the function $u_m \in X_2$ given by 
\begin{equation}
  \label{eq:def-u-n}
u_m(t,x)=U_m(|t|) \qquad \text{with}\quad U_m(r)=\dfrac{I_{N/2-1}(j_mr)}{I_{N/2-1}(j_m)}, 
\end{equation}
see (\ref{eq:def-tilde-un}). We also set
\begin{equation}
  \label{eq:def-cU}
\cU_0:=\{h \in C^{2,\alpha}_{p}(\R) \::\: h>-1\}
\end{equation}
and define the operator 
\begin{align}\label{eq:MappinmgF-ND}
F_\lambda:   X_2^{DN} \times \cU_0 \to Y, \qquad F_\lambda(u, h)= L_\lambda^{1+   h} (u+u_m),
  \end{align}
By construction, we then have the following:
{\em if $F_\lambda(u, h)=0$, then the function $\tilde u = u_m + u$ solves the problem
\begin{equation}
  \label{eq:equivalence-F-ND}
  \left \{
    \begin{aligned}
         L_\lambda^{1+  h} \ti u & =  0 && \qquad \text{in $\Omega_*$,}\\
           \ti u&= 1 &&\qquad \text{on $\partial \Omega_*$,}\\
          \n \ti {u} \cdot \t&= 0 &&\qquad \text{on $\partial \Omega_*$.}
    \end{aligned}
       \right.
\end{equation}}
We need to check first, however, that $F$ is well-defined by (\ref{eq:MappinmgF-ND}). This is done by the following lemma.

\begin{Lemma}
  \label{sec:functional-setting-1}  The map
  $$
  (u,h) \mapsto F_\lambda(u, h)= L_\lambda^{1+   h} (u+u_m)
  $$
  maps $X_2^{DN} \times \cU_0$ into $Y$.
\end{Lemma}

\begin{proof}
Using (\ref{eq:reldiffope-ND-alt}) and the fact that $u_m$ does not depend on $x$, we may write $F_\lambda = F_\lambda^1 + F_\lambda^2$
with
\begin{align}\label{eq:MappinmgF-ND1}
  &F_\lambda^1(u,h)=j_m^2( u + u_m) +  (1+h)^2|t|^2 \D_t(u+u_m)   \nonumber\\
  &\qquad \quad \;\;\,+ \lambda     \frac{(h')^2}{(1+h)^2}{D}_{t} D_{t}(u+u_m) + 2 \lambda  \frac{ h' }{1+h}   {D}_t \partial_{ x  }u
\end{align}
and
\begin{align}\label{eq:MappinmgF-ND2}
F_\lambda^2(u,h) = \lambda \partial_{xx}u
 +  (1+h)^2(1-|t|^2) \D_t(u+u_m) + \lambda  \Bigl(\frac{ h''}{h}-\frac{h'^2}{h^2}\Bigr){D}_t(u+u_m). 
\end{align}
We first note that
\begin{equation}
  \label{eq:F-lambda-2-prop}
\text{$F_\lambda^1$ maps $X_2^{DN} \times \cU_0$ into $C^{1,\alpha}_{p,rad}(\overline \Omega_*)$.} 
\end{equation}
Indeed, if $u \in X_2^{DN}$, then $u + u_m,\: {D}_t(u+u_m) \in C^{2,\alpha}_{p,rad}(\overline \Omega_*)$ by definition of the space $X_2^{DN}$ and of the function $u_m$. In particular, ${D}_{t} D_{t}(u+u_m) \in C^{1,\alpha}_{p,rad}(\overline \Omega_*)$. Moreover, it follows from Remark~\ref{sec:functional-setting-2} that ${D}_t \partial_{ x  }u= \partial_x {D}_t u \in C^{1,\alpha}_{p,rad}(\overline \Omega_*)$.
Finally, we note that
\begin{equation}
  \label{Delta-t-C-1-alpha}
|t|^2 \D_t(u+u_m) \in C^{1,\alpha}_{p,rad}(\overline \Omega_*).
\end{equation}
In fact, this follows since $w:=u+u_m$ is radial in $t$ and therefore 
\begin{equation}
\label{computation-t-Laplace}
\begin{split}
  |t|^2 \D_t w(t,x) &= \nabla^2_t w(t,x)[t,t] +(N-1) \nabla w(t,x) \cdot t\\
  &= {D}_t{D}_tw(t,x) + (N-2)D_t w(t,x),
\end{split}
\end{equation}
where we used (\ref{nabla-delta-comp-1}) in the last step. With this representation, (\ref{Delta-t-C-1-alpha}) follows since ${D}_tw \in C^{2,\alpha}_{p,rad}(\overline \Omega_*)$ by definition of $u_m$ and the space $X_2^{DN}$.
From the facts collected above we deduce that $F_\lambda^1(u,h) \in C^{1,\alpha}_{p,rad}(\overline \Omega_*)$ if $h \in \cU_0$, and thus (\ref{eq:F-lambda-2-prop}) follows.\\ 
Next we show that
\begin{equation}
  \label{eq:F-lambda-2-prop-1}
\text{$F_\lambda^2$ maps $X_2^{DN} \times \cU_0$ into $X_0^D$.} 
\end{equation}
To see this, let $u \in X_2^{DN}$. Then $u \equiv 0$ on $\partial \Omega_*$ implies that $\partial_{xx} u \equiv 0$ on $\partial \Omega_*$. Moreover, $\partial_{xx} u \in C^{0,\alpha}(\overline \Omega_*)$ and $D_\t \partial_{xx} u= \partial_{xx} D_\t u \in C^{0,\alpha}(\overline \Omega_*)$ by Remark~\ref{sec:functional-setting-2}. Hence $\partial_{xx}u \in X_0^D$.
Moreover, we have ${D}_tu \equiv 0$ on $\partial \Omega_*$ by definition of $X_2^{DN}$ and ${D}_tu_m \equiv 0$ on
$\partial \Omega_*$ by the definition of $u_m$. In addition, ${D}_t(u+u_m) \in C^{0,\alpha}$ and  ${D}_t {D}_t(u+u_m) \in C^{0,\alpha}(\overline \Omega_*)$ since $u,u_m \in C^{2,\alpha}(\overline \Omega_*)$. Hence ${D}_t(u+u_m) \in X_0^D$.
Finally, we note that 
\begin{equation}
  \label{Delta-t-C-1-alpha-2}
(1-|t|^2) \D_t(u+u_m) \in X_0^D.
\end{equation}
Obviously, we have $(1-|t|^2) \D_t(u+u_m) \equiv 0$ on $\partial \Omega_*$. Moreover, $(1-|t|^2) \D_t(u+u_m) \in C^{0,\alpha}(\overline \Omega_*)$ since $u,u_m \in C^{2,\alpha}(\overline \Omega_*)$. Finally, a direct computation shows that for $w= u+u_m$ we have 
$$
{D}_t \D_t w = \Delta_t \bigl({D}_t w - 2 w\bigr),
$$
whereas ${D}_t w - 2 w \in C^{2,\alpha}(\overline \Omega_*)$ by the definition of the space $X_2^{DN}$. Hence $
{D}_t \D_t w \in C^{0,\alpha}(\overline \Omega_*)$ and thus also ${D}_t\Bigl( (1-|t|^2) \D_t(u+u_m)\bigr) \in C^{0,\alpha}(\overline \Omega_*)$. This shows (\ref{Delta-t-C-1-alpha-2}). We now conclude that $F_\lambda^2(u,h) \in X_0^D$ if $h \in \cU_0$, and thus (\ref{eq:F-lambda-2-prop}) follows.

The claim of the lemma now follows from (\ref{eq:F-lambda-2-prop}), (\ref{eq:F-lambda-2-prop-1}) and the definition of the space $Y$.
\QED \end{proof}

We now wish to reduce \eqref{eq:equivalence-F-ND} further to a problem depending only on the variable $u$ and not on $h$ anymore. A motivation for the following reduction procedure will be given in Remark~\ref{rem:const-sol} below. We define the linear map
\begin{equation}
\label{def-M}
M: X_2^D \to X_2^{DN} \times  C^{2,\alpha}_{p}(\R),\qquad  M u = (M_1 u,h_u)
\end{equation}
with
\begin{equation}\label{eq:DeffM-ND--2}
h_u(x)= \frac{I_{N/2-1}(j_m)}{j_m^2 I_{N/2-1}''(j_m)} {D}_t u (e_1,x), \qquad x \in \R
\end{equation}
and
\begin{equation}\label{eq:DeffM-ND--1}
[M_1 u](t,x)= u(t,x) - g(|t|) h_u(x), \qquad (t,x) \in \Omega_*,
\end{equation}
where $e_1 = (1,0,\dots,0)$ is the first coordinate vector in $\R^N$ and we define 
\begin{equation}
  \label{eq:def-g_n}
g \in C^\infty([0,\infty)),\qquad  g(r) = r U_m'(r)= \frac{j_{m} r I_{N/2-1}'(j_{m} r)}{I_{N/2-1}(j_m)}= - j_{m}^2 r^2 \frac{I_{N/2}(j_{m} r)}{I_{N/2-1}(j_m)}. 
\end{equation}
Here we used (\ref{eq:derrec}) for the last equality, and $U_m$ is defined as before by $u_m(t,x) = U_m(|t|)$, see (\ref{eq:def-u-n}). We easily see that
\begin{equation}
  \label{eq:g-properties}
g'(0)=0, \qquad g(1)=0 \qquad \text{and}\qquad g'(1)= \frac{j_m^2 I_{N/2-1}''(j_m)}{I_{N/2-1}(j_m)} . 
\end{equation}
From this we deduce that indeed we have $M_1 u \in X_2^{DN}$ for $u \in X_2^D$. Moreover, $h_u \in  
C^{2,\alpha}_{p}(\R)$ for $u \in X_2^D$ by definition of $X_2^D$. Hence the linear map $M$ is well defined by (\ref{def-M}). Next we note the following.

\begin{Lemma}
  \label{M-isomorphism}
The map $M$ defined by (\ref{def-M}) is a topological isomorphism $X_2^D \to X_2^{DN} \times  C^{2,\alpha}_{p}(\R)$.  
\end{Lemma}

\begin{proof}
  By definition, $M$ defines a bounded linear map $X_2^D \to X_2^{DN} \times  C^{2,\alpha}_{p}(\R)$. Moreover, $M$ is bijective with inverse
$$
N:  X_2^{DN} \times  C^{2,\alpha}_{p}(\R) \to X_2^D, \qquad [N(u,h)](t,x)=u(t,x)+g(|t|)h(x).
$$
Since both $X_2^D$ and $X_2^{DN} \times  C^{2,\alpha}_{p}(\R)$ are Banach spaces, the claim now follows from the open mapping theorem.
\QED \end{proof}

With the help of the linear map $M$, we may now reduce our problem further and define
\begin{equation}\label{eq:DeffGl-ND}
G_\lambda: \cU  \to Y,\qquad G_\lambda = F_\lambda \circ M
\end{equation}
with 
$$
\cU:= \left\{u \in X_2^D\::\: \, h_u(x)  >-1 \;\text{for $x \in \R$} \right \}.
$$
We then note the following equivalences:
\begin{align}
  G_\lambda(u)= 0\quad &\Longleftrightarrow \quad F_\lambda(M_1 u, h_u)=0\nonumber\\
  &\Longleftrightarrow \quad L_\lambda^{1+ h_u} (M_1 u+u_m)  =  0 \qquad \text{in $\Omega_*$}\nonumber\\
  &\Longleftrightarrow \quad \text{$M_1 u +u_m$ solves (\ref{eq:equivalence-F-ND}) with $h= h_u $.} \label{eq:Eqtosolve-ND}
  \end{align}
In the last step, we used the fact that $M_1 u \in X_2^{DN}$.\\ 

\begin{Remark}\label{rem:const-sol}
 
{\rm (i)   In principle, any function $g$ satisfying (\ref{eq:g-properties}) could be used to define the linear transformation $M$. However, we shall see in Proposition~\ref{sec:functional-setting-4} below that, with the choices made above, the derivative $D G_\l(0)$ of the map $G_\l$ defined in (\ref{eq:DeffGl-ND}) has a particularly nice form.\\ 
(ii) Related to this fact, we now provide an intrinsic motivation for the ansatz made in \eqref{eq:Eqtosolve-ND}. For this we note that, extending the function $u_m$ to all of $\R^N \times \R$ by definition~(\ref{eq:def-u-n}), we have $L_\l u_m=0$ in $\R^N \times \R$ and therefore, for fixed $h \in C^{2,\alpha}_{p}(\R)$, 
  \begin{equation}
\label{eq:perturbed-u-n-h}
L^{1+ h}_\l (u_m^h)= 0 \qquad \text{with $u_m^h \in C^{2,\alpha}(\ov{\Omega_*}),\quad$ $u_m^h(t,x)= u_m(\frac{t}{1+h(x)},x)$}
\end{equation}
by \eqref{eqreladiffopets-ND}. Moreover,
\begin{equation}
\label{eq:approximation-u-n-h}
u_m^h = u_m - w_h + O(\|h\|_{C^{2,\alpha}}^2),
\end{equation}
where
\begin{equation}
\label{eq:approximation-u-n-h-1}
w_h(t,x):= {D}_t u_m(t,x) h(x)= |t|U_m'(|t|) h(x)= g(|t|)h(x) 
\end{equation}
This shows that  $u_m-w_h$ is a natural approximate solution for any small $h$.  We can thus   look for a solution to (\ref{eq:equivalence-F-ND}) of the form $\tilde u:=u_m-w_h +u$,  with $u $  and $h$ small. Assuming here that $u=0$ on $\de\O_*$, we have $\tilde u \equiv 1$ on $\de\O_*$ by the definition of $u_m$ and since $g(1)=0$. Moreover, since ${D}_t  u_m \equiv 0$ on $\de\O_*$, the condition ${D}_t  \tilde u \equiv 0$ on $\de\O_*$ enforces 
$$
{D}_t u(e_1,x)= {D}_t w_h(e_1,x) = g'(1) h(x) 
$$
and therefore 
$$
h(x)= \frac{{D}_t u (e_1,x)}{g'(1)} =h_u(x)  \qquad \text{for $x \in \R$.}
$$
This is what motivates the structure of the solution as described by \eqref{eq:DeffM-ND--2}, \eqref{eq:DeffM-ND--1} and  \eqref{eq:Eqtosolve-ND}.} 
\end{Remark}\vspace{0.5ex}

Next we note that, by definition of the map $G_\lambda$ in (\ref{eq:DeffGl-ND}), we have  
\begin{equation}
  \label{eq:trivial-branch}
G_\lambda(0) = 0 \qquad \text{for all $\lambda >0$.}
\end{equation}
To check if this trivial branch of solutions admits bifurcation, we need to consider the derivative $D G_\lambda(0)$.  We have the following result.

\begin{Proposition}
  \label{sec:functional-setting-4}
  The map $G_\l :\cU \subset X_2^D \to Y $ defined by  \eqref{eq:DeffGl-ND} is of class $C^\infty$. Moreover for all $v\in X_2^D$,
\begin{align}\label{eequivla-ND}
D G_\lambda(0)v= L_\l v=j_{m}^2v+ \lambda v_{xx} + \D_t v.
\end{align}
\end{Proposition}

\begin{proof}
  To see that $G_\l$ is of class $C^\infty$, by the linearity of the operator  $M$, we only need to check that $F_\l:X_2^{DN} \times \cU_0 \to Y$ is of class $C^\infty$.  Now from the definition of the norm in $Y$, it is also easy to see that  $F_\l$ is of class $C^\infty$ as soon as the maps
  $$
  F_\lambda^1: X_2^{DN} \times \cU_0 \to C^{1,\alpha}_{p,rad}(\overline \Omega_*) \;\subset\; Y
  $$
  and
  $$
  F_\lambda^2: X_2^{DN} \times \cU_0 \to X_0^{D} \;\subset\; Y
  $$
defined in (\ref{eq:MappinmgF-ND1}) and (\ref{eq:MappinmgF-ND2}) are of class $C^\infty$, and in this case $D^k F_\l=D^kF^1_\l+D^kF^2_\l$ for all $k\in \N$.  Now the $C^\infty$-regularity $F_\l^i$, $i=1,2$ follows clearly from the definitions of these maps.

To see (\ref{eequivla-ND}), we start by differentiating (\ref{eq:perturbed-u-n-h}) to get, for fixed $h \in 
C^{2,\alpha}_{p}(\R)$
\begin{align}
  0 = \frac{d}{ds}\Bigl|_{s=0} \Bigl(L^{1+ sh}_\l (u_m^{sh})\Bigr) 
    &= \Bigl(\frac{d}{ds}\Bigl|_{s=0} L^{1+ sh}_\l\Bigr)u_m + L^1_\l \frac{d}{ds}\Bigl|_{s=0}u_m^{sh} \nonumber\\
  &= \Bigl(\frac{d}{ds}\Bigl|_{s=0} L^{1+ sh}_\l\Bigr)u_m -  L_\l w_h  \label{zero-diff}  
\end{align}
with $w_h(t,x)= g(|t|)h(x)$, where we used (\ref{eq:approximation-u-n-h}) in the last step.
By the chain rule, we now have
\begin{equation}
DG_\l(0)v=\de_u F_\l(0,0)M_1 v+ \de_h F_\l(0,0)h_v \qquad \text{for $v \in X_2^D$,}
\label{eq:DGl0-compt}
\end{equation}
where, since by definition $M_1 v = v - w_{h_v}$ with $w_{h_v}(t,x) = g(|t|)h_v(x)$,  
$$
\de_u F_\l(0,0)M_1 v = L_\lambda M_1 v =L_\lambda v - L_\lambda w_{h_v}
$$
and, by (\ref{zero-diff}), 
$$
\de_h F_\l(0,0)h_v = \Bigl(\frac{d}{ds}\Bigl|_{s=0} L^{1+ sh_v}_\l\Bigr)u_m = L_\l w_{h_v}.
$$
Combining these identites gives $DG_\l(0)v=L_\lambda v$ for $v \in X_2^D$, as claimed.
\QED\end{proof}

In order to apply bifurcation theory in the functional analytic setting proposed in this section, it is important to observe that $L_\lambda = DG_\lambda : X_2^D \to Y$ is a Fredholm operator for every $\lambda >0$. We shall prove this in Proposition~\ref{fredholm} below, and the following lemma is a key step towards this property. 

\begin{Lemma}
  \label{regularity-lemma-1-ND} 
Let $f \in C^{0,\alpha}_{p,rad}(\ov{\Omega_*})$ and let $u \in C^{2,\alpha}_{p,rad}(\ov{\Omega_*})$ satisfy 
  \begin{equation}
    \label{eq:regularity-x0-lemma-ND}
 \D_t u + \lambda u_{xx} = f \quad \text{in $\Omega_*$,}\qquad u = 0 \quad \text{on $\partial \Omega_*$.}   
  \end{equation}
If $f \in Y$, then $u \in X_2^D$. 
\end{Lemma}

\begin{proof}
  Since $u \in C^{2,\alpha}(\ov{\Omega_*})$, we have 
  $v:={D}_t u \in C^{1,\alpha}_{p,rad}(\ov{\Omega_*}).$  We claim that $v \in C^{2,\alpha}_{loc}({\Omega_*})$, and that $v$ satisfies the equation 
   \be\label{eq:Dtuo}
  \D_t v+ \l \de_{xx} v=- 2\l \de_{xx} u+2 f+ {D}_t f  \qquad\textrm{ in $\O_*$.}
  \ee
  Note here that ${D}_t f \in C^{0,\alpha}(\ov{\Omega_*})$ exists since $f \in Y$.  Hence the RHS of \eqref{eq:Dtuo} belongs to $C^{0,\alpha}(\overline{\O_*})$, and therefore, by Schauder regularity theory, it suffices to establish this equation in distributional sense. For $\psi \in C^\infty_c(\Omega_*)$ we have
  $$
  \int_{\Omega_*} v [\Delta_t+ \l \de_{xx}]\psi(t,x) d(t,x)  = \frac{d}{d\s}\Big|_{\s=1} \int_{\Omega_*}u(\s t,x)[\Delta_t+ \l \de_{xx}]\psi(t,x) d(t,x)
  $$
  where by (\ref{eq:regularity-x0-lemma-ND}), for $\sigma$ sufficiently close to $1$, 
  \begin{align*}
    \int_{\Omega_*}&u(\s t,x) [\Delta_t + \l \de_{xx}]\psi(t,x)]d(t,x)= \int_{\Omega_*}[\sigma^2 \Delta_t u(\s t,x) + \l \de_{xx}u(\s t,x)]\psi(t,x) d(t,x)\\
    &=\int_{\Omega_*}[\sigma^2 f(\s t,x)  - \lambda (\s^2-1) \de_{xx} u(\s t,x)]\psi(t,x) d(t,x) \qquad\textrm{ in $\overline \O_*$}
  \end{align*}
  and thus
  \begin{align*}
    \frac{d}{d\s}\Big|_{\s=1}
    \int_{\Omega_*}&u(\s t,x) [\Delta_t + \l \de_{xx}]\psi(t,x)]d(t,x)\\
                   &=\int_{\Omega_*}[2 f + {D}_t f]\psi(t,x) d(t,x) - \lambda \lim_{\s \to 1}\frac{\s^2-1}{\s-1}
                     \int_{\Omega_*}\de_{xx}u(\s t,x) \psi(t,x) d(t,x)\\
   &=\int_{\Omega_*}[2 f + {D}_t f-2\lambda u_{xx} ]\psi d(t,x).
  \end{align*}
  Here we used that ${D}_t f$ exists in $C^{0,\alpha}(\ov{\Omega_*})$ by assumption. Hence we have established \eqref{eq:Dtuo} in distributional sense, as required.

  Next, we note, using the identity~(\ref{computation-t-Laplace}) for radial functions in $t$,  that
 $$
{D}_t v =   {D}_t {D}_t u  = |t|^2 \Delta_t u  -(N-2){D}_t u  \qquad\textrm{ in $\overline \O_*$,}
  $$
where  
  $$
  \D_t u=- \lambda \de_{xx} u + f = f \qquad \text{on $\partial \Omega_*$,}
  $$
  since $u \equiv 0$ on $\partial \Omega_*$ by assumption. Hence 
   \be\label{eq:Dtuo-1}
   {D}_t v 
=  f- (N-2){D}_t u \qquad \text{on $\partial \Omega_*$.}
  \ee
  Here we note that ${D}_t u\big|_{\partial \O_*} \in C^{1,\alpha}(\de{\O_*})$ since $u \in C^{2,\alpha}(\ov{\Omega_*})$. Moreover, writing
  $$
  f = f_1 + f_2 \in Y\qquad \text{with $f_1 \in C^{1,\alpha}_{p,rad}(\ov{\O_*})$ and $f_2 \in X_0^D$,}
  $$
  we see that $f\big|_{\partial \O_*}= f_1 \big|_{\partial \O_*} \in C^{1,\alpha}(\de{\O_*})$ since $f_2 \big|_{\partial \O_*} \equiv 0$ by definition of $X_0^D$. Hence the RHS of \eqref{eq:Dtuo-1} belongs to $C^{1,\alpha}(\de{\O_*})$. Consequently, standard elliptic regularity, applied to the Neumann boundary value problem \eqref{eq:Dtuo}, \eqref{eq:Dtuo-1}, gives $v={D}_t u \in C^{2,\alpha}(\overline{\O_*})$, as required. \QED 
\end{proof}

\begin{Proposition}
\label{fredholm}  
For every $\lambda>0$, the operator $L_\lambda = DG_\lambda : X_2^D \to Y$ is a Fredholm operator of index zero.
\end{Proposition}

\begin{proof}
  Consider the operator $\widetilde{L}: X_2^D \to Y$, $\widetilde{L} v= \lambda v_{xx} + \D_t v$. We show that 
  \begin{equation}
    \label{topol-iso-tilde-L}
   \text{$\widetilde{L}$ is a topological isomorphism.}  
  \end{equation}
  Once this is shown, the claim follows since the difference operator $L_\lambda- \widetilde{L}= j_m^2 \id: X_2^D \to Y$ is compact (i.e., the embedding $X_2^D \hookrightarrow Y$ is compact) and since the Fredholm property and the Fredholm index are stable under compact perturbations.
  
By the open mapping theorem, we only have to show that $\widetilde{L}$ is bijective to deduce (\ref{topol-iso-tilde-L}). To see injectivity, let $u \in X_2^D$ satisfy
  $$
  \widetilde{L}u = \lambda u_{xx} + \D_t u= 0 \qquad \text{in $\Omega_*$.}
  $$
  Multiplying this equation by $u$ and integrating by parts in $t$ and in $x$ from $0$ to $2\pi$ gives
  $$
  - \int_{\Omega_*}\Bigl( \lambda |u_x|^2 + |\nabla_t u|^2)\,d(t,x) = 0,
  $$
  hence $u$ is a constant function and therefore $u \equiv 0$ since $u \in X_2^D$.

  Next we show that $\widetilde{L}$ is onto. For this we let $f \in Y \subset C^{0,\alpha}_{p,rad}(\ov{\Omega_*})$. Since the operator $\lambda \partial_{xx} + \D_t$ is uniformly elliptic, it is well known that the problem
\begin{equation}
  \label{g-eq}
\lambda u_{xx} + \D_t u =f \quad \text{in $\Omega_*$,}\qquad u= 0 \quad \text{on $\partial \Omega_*$}
\end{equation}
  has a solution $u \in C^{2,\alpha}_{p,rad}(\ov{\Omega_*})$. Moreover, $u \in X_2^D$ and $  \widetilde{L} u = f$ by Lemma~\ref{regularity-lemma-1-ND}. This shows that $  \widetilde{L}$ is onto, and the proof is finished.
\QED
\end{proof}

\begin{Remark}\label{rem-loss}
{\rm We summarize briefly why we have introduced the particular functional analytic framework of the present section. In principle, $G_\l:C^{3,\alpha}_{p,rad}(\ov{\Omega_*})\to C^{0,\alpha}_{p,rad}(\ov{\Omega_*})$ is a well defined smooth map and $DG_\l(0)=L_\l$.   This  can be easily  seen from \eqref{eq:reldiffope-ND},  \eqref{def-M} and the definition of $h_u$ in \eqref{eq:DeffM-ND--2}.
It is however not true that $ L_\l:C^{3,\alpha}_{p,rad}(\ov{\Omega_*})\cap X_0^D\to C^{0,\alpha}_{p,rad}(\ov{\Omega_*})$ is  Fredholm operator of index zero because it is a second order elliptic operator.     This results in a loss of regularity  which we have overcome by working in the tailor made spaces $X_2^D$ and $Y$. We also notice, in particular,  that  the space $Y$ is mainly determined by the structure of the operator $L_\l^h$ and the structure of the solutions to the equation $F_\l=0$ given by  $(M_1u+u_m,1+h_u)$.}
\end{Remark}
\vspace{1ex}

\section{The Linearized operator  $D G_\lambda(0)$  and its spectral properties }\label{eq:kernel}

Let, as before, $J_{\nu}$ denote the Bessel function of the first kind of order $\nu>-1$, $I_\nu(r)= r^{-\nu} J_\nu(r)$ for $r>0$, and let 
$$
0<j_{\nu,1}<j_{\nu,2}<\dots
$$
 be the  zeros of $J_\nu$.  Recall also that we use the simplified notation $j_n:= j_{\frac{N}{2},n}$ for the special case $\nu = \frac{N}{2}$. We first consider the ODE eigenvalue problem
\begin{equation}
  \label{eq:v-k-initial-vu-ND}
- u'' -  \frac{N-1}{r}u' = - \frac{1}{r^{N-1}}(r^{N-1}u')' = \mu u \quad\text{in (0,1)}, \qquad \; u'( 0)=u( 1)=0  
\end{equation}
satisfied by radial eigenfunctions of the Dirichlet Laplacian in the unit ball. It is well known that the eigenvalues of (\ref{eq:v-k-initial-vu-ND}) are given by
$$
j^2_{N/2-1,n},\qquad n \geq 1
$$
with associated eigenfunctions
$$
r \mapsto \vp_n(r)=I_{N/2-1}(j_{N/2-1, n}\,r ).
$$
We note the important property
$$
 j_{\nu,n}<j_{\nu+1,n}<j_{\nu, n+1} \qquad \text{for $\nu>-1$, $n\geq 1$} 
$$
(see e.g. \cite[Chapter XV, 15.22]{NevWatson}). For the special case $\nu=N/2-1$, we thus get the following inequalities involving the eigenvalues of (\ref{eq:ODE-eigenvalue}) and (\ref{eq:v-k-initial-vu-ND}):
\be\label{eq:sep-zeros-Bessel}
 j_{N/2-1, n} < j_n  < j_{N/2-1,n+1} \qquad \text{for $n \geq 1$.}
\ee
As in the previous sections,  we now fix $m   \geq 1$ and consider the operator $L_\lambda =DG_\lambda(0)$ given in Proposition~\ref{sec:functional-setting-4}. Moreover, we define the value
\be \label{eq:ei}
 \l_m:= {j_{m}^2-j_{N/2-1,1}^2 }, 
 \ee
 which is positive by \eqref{eq:sep-zeros-Bessel}. The following proposition establishes key properties of the operator $L_{\lambda_m}$ which will allow us to apply the Crandall-Rabinowitz bifurcation theorem.

\begin{Proposition}\label{propCR-ND}
\begin{itemize}
\item[(i)] The kernel $N(L_{\l_m})$ of the operator
$$
L_{\l_m}: X^D_2  \to Y, \qquad L_{\l_m} v =j_{m}^2v+ \lambda_m v_{xx} + \D_t v,
$$
is one-dimensional and spanned by the function
\begin{equation}\label{eq:kernel-ND}
 v_{*}(t,x)= I_{N/2-1}(j_{N/2-1, 1}|t|)\cos( x).
\end{equation}
\item[(ii)] The image $R(L_{\l_m})$ of $L_{\l_m}$ is given by 
  $$
  R(L_{\l_m} )= \left \lbrace  w\in Y: \int_{\O_* } v_{*}(t,x)w(t,x)\,dxdt=0\right\rbrace.
$$
\item[(iii)] We have the transversality property 
\begin{equation}
  \label{eq:transversality-cond4-ND}
\partial_\lambda \Bigl|_{\lambda=\l_{m} }L_\lambda v_{*}\not  \in \; R(L_{\l_m} ).\\
\end{equation}
\end{itemize}
\end{Proposition}

\begin{proof}
  (i) Let $v \in N(L_{\l_m})$. So $v\in X_2^D$ satisfies $L_{\l_m} v = 0$, i.e., 
\begin{align}\label{eq. for U-ND} 
\D_tv +\lambda_m v_{xx}  + j_m^2  v=0 \qquad \text{in $\Omega_*$}
\end{align}
together with the condition
\begin{align}\label{eqbounU-ND}
v=0 \qquad \text{on $\partial \Omega_*$.}
\end{align}
By elliptic regularity theory, $v\in C^\infty(\ov{\O_*})$. Since, moreover, the function $v$ is radial in $t$ and $2\pi$ periodic  and even in $x$, we have a Fourier series expansion $v(t,x)=\sum_{\ell=0}^\infty v_\ell(|t|)\cos(\ell x)$ in the $x$-variable, where the $|t|$-dependent Fourier coefficients are given by    
$
v_\ell(|t|):= \frac{1}{\sqrt{2\pi}}\int_{0}^{2\pi} v(t,x)\cos (\ell x)\,dx.
$
Since the Fourier series expansion given above holds in $C^2$-sense, we may multiply  \eqref{eq. for U-ND}  and \eqref{eqbounU-ND} with $\cos (\ell x)$ and integrate in the $x$-variable from $0$ to $2\pi$ to get 
\begin{equation}
  \label{eq:ODE-ev-ell}
-v_\ell'' -\frac{N-1}{r} v_\ell' =(j_m^2-\ell^2 \lambda_m) v_\ell  \quad\textrm{on (0,1)}, \qquad v_\ell'(0)=v_\ell( 1)=0  .
\end{equation}
If $v_\ell \not = 0$, then, by the remarks in the beginning of this section, we must have $j_m^2- \ell^2\lambda_m= j^2_{N/2-1,k}$,  for some $k\geq 1$, and therefore      
\begin{align}\label{condikerne}
\ell^2 ({j_{m}^2-j_{N/2-1,1}^2 } )=\ell^2 \lambda_m =  j_m^2- j^2_{N/2-1,k}.
 \end{align}
 By \eqref{eq:sep-zeros-Bessel} (applied with $m=n$), this is only possible if $k=1$ and $\ell = 1$, and in this case the space of solutions of (\ref{eq:ODE-ev-ell}) is spanned by 
$$
r \mapsto I_{N/2-1}(r j_{N/2-1, 1}).
$$
Consequently, $v= A v_*$ with some constant $A \in \R$ and $v_*$ given in (\ref{eq:kernel-ND}). It is also easily seen that, with this definition, $v_*$ satisfies $L_{\l_m}v_* =0$.  Thus (i) is proved.

To prove (ii), we first note that $R(L_{\l_m})$ has codimension one by (i) and since $L_{\l_m}$ is Fredholm of index zero by Proposition~\ref{fredholm}. Hence we only need to prove that 
$$
R(L_{\l_m}) \subset \left \lbrace  w\in Y: \int_{\O_* } v_{*}(t,x)w(t,x)\,dxdt=0\right\rbrace.
$$
This follows immediately since, via integration by parts,
$$
\int_{\O_* } v_{*} L_{\l_m} u \,dxdt= \int_{\O_* } u L_{\l_m} v_{*} \,dxdt = 0 \qquad \text{for every $u \in X^D_2.$}
$$
Finally, (iii) follows from (ii) and the fact that   
$$
\de_\l\big|_{\l=\l_{m}}DG_\l(0) v_{*}= \de_{xx} v_{*}=-v_{*}.
$$
   \QED
\end{proof}

\section{Proof of Theorem \ref{Theo1-ND}  }\label{eq:ProofTheo1-ND}

As before, we fix $m \geq 1$, and we let $G_\lambda$ be defined as in \eqref{eq:DeffGl-ND}.
The proof of Theorem  \ref{Theo1-ND}  is obtained by an application
of the Crandall-Rabinowitz Bifurcation theorem to solve the equation
\begin{align}\label{eq:maptsolvG}
 G_\lambda(u)=0,
\end{align}
This application gives, in a first step, the following result. 

\begin{Theorem}\label{Theo1-general-ND}
There  exist  ${\e}>0$ and a smooth curve
$$
(-{\e},{\e}) \to   (0,+\infty) \times X_2^D ,\qquad s \mapsto (\lambda(s) ,\varphi_{s})
$$
with
\be
G_{\l(s)}(\varphi_{s}) =0.
\ee
Moreover, $\lambda(0)= \l_m$ (see \eqref{eq:ei}) and 
$$
\varphi_{s} = s(v_{*}+\o_{s}) \qquad \text{for $s \in (-{\e},{\e})$,}
$$
with $v_{*}$ given in (\ref{eq:kernel-ND}) and a smooth curve 
$$
(-{\e},{\e}) \to X_2^D, \qquad s \mapsto \o_{s}
$$
satisfying  $\o_{0} =0$ and 
\begin{align*}
\int_{\O_*}  \o_{s} (\t,x)  v_{*}(\t,x)\,dxd\t=0 \qquad \text{for $s \in (-{\e},{\e})$.}
\end{align*}
\end{Theorem}

\begin{proof}
  The claim is a direct consequence of the Crandall-Rabinowitz Theorem (see \cite[Theorem 1.7]{M.CR}) applied to the map
  $$
  (-\l_m,\infty) \times \calU \to Y,\qquad (\lambda,u) \mapsto  G_{\lambda+\l_m}(u)
  $$
  The assumptions of \cite[Theorem 1.7]{M.CR} are satisfied by (\ref{eq:trivial-branch}), Proposition~\ref{sec:functional-setting-4} and Proposition \ref{propCR-ND}.
\QED \end{proof}

\subsection*{Proof of Theorem \ref{Theo1-ND} (completed)}
\label{sec:proof-theor-refth}
We continue using the notation from Theorem~\ref{Theo1-general-ND}.
Recalling  \eqref{eq:Eqtosolve-ND}, we see that, since  $G_{\lambda(s)} (\vp_{s})=0$ for every  $s \in (-{\e},{\e})$, the function 
\begin{align}
 \ti u_s ( \t,x)&:= u_m(\t,x) + [M_1   \varphi_{s}]({\t},x) \nonumber\\
  &= u_m(\t,x) + \varphi_{s}({\t},x)-g(|t|)h_{\vp_{s}}(x)  \label{eq:Solutionfinal} 
\end{align}
solves  (\ref{eq:equivalence-F-ND}) with
\begin{equation}
  \label{eq:def-h-v-phi-s}
h_{\vp_{s}}(x) = \frac{I_{N/2-1}(j_{m})}{j_m^2 I_{N/2-1}''(j_{m})}{D}_t \vp_{s}(e_1,x)
\end{equation}
and $g(r)=r U_m'(r)$ as defined in (\ref{eq:DeffM-ND--2}) and (\ref{eq:def-g_n}). Consequently, by the scaling and transformation properties discussed in the beginning of Section~\ref{eq.setting}, the function 
$$
(t,x )\mapsto w_s(t,x)=\ti u_s(h_{s}(x)t,x)
$$
solves \eqref{eq:solved-main-ND} with $\mu_s=\frac{j^2_{m}}{\l(s)} $ and 
$$
h_{s}(x)=\frac{1+h_{\varphi_{s}}(x)}{\sqrt{\l(s)}} \qquad \text{for $x \in \R$.}
$$
Moreover, by Theorem~\ref{Theo1-general-ND},
\begin{equation}
  \label{eq:value-lambda-0}
\l(0)= \l_m= j_{m}^2-j_{N/2-1,1}^2 
\end{equation}
and
\begin{equation}
  \label{eq:varphi-s-exp}
\varphi_{s}(t,x)= s v_*(t,x)+ o(s)= s I_{N/2-1}(j_{N/2-1,1}|t|)\cos( x)+ o(s),
\end{equation}
where $o(s) \to 0$ in $C^2$-sense in $\ov{\Omega_*}$ as $s \to 0$. Hence, by (\ref{eq:derrec}),  
\begin{equation*}
  \begin{split}
    D_t \varphi_{s}(e_1,x)&= s j_{N/2-1,1} I_{N/2-1}'(j_{N/2-1,1})\cos( x)+ o(s)\\
    &=- s j_{N/2-1,1}^2 I_{N/2}(j_{N/2-1,1})\cos( x)+ o(s),
  \end{split}
\end{equation*}
where $o(s) \to 0$ in $C^1$-sense in $\ov{\Omega_*}$ as $s \to 0$. Using this, we find that 
\begin{equation}
\begin{split}
  h_{\varphi_{s}} (x)&= \frac{I_{N/2-1}(j_{m})}{j_m^2 I_{N/2-1}''(j_{m})} D_t \varphi_{s}(e_1,x)\\
                  &= - s\, \frac{j_{N/2-1,1}^2I_{N/2-1}(j_{m})  I_{N/2}(j_{N/2-1,1}) }{ j^2_{m}I_{N/2-1}''(j_{m})}\cos( x)+ o(s) \qquad \text{as $s \to 0$}
\end{split}
  \label{eq:varphi-s-der-exp}
                    \end{equation}
and therefore, by (\ref{eq:value-lambda-0}), 
\begin{align*}
  h_{s}(x)&=\frac{1+h_{\varphi_{s}}}{\sqrt{\l(s)}}= \frac{1}{\sqrt{\l(s)}} -  s\frac{j_{N/2-1,1}^2 I_{N/2-1}(j_{m}) I_{N/2}(j_{N/2-1,1}) }{\sqrt{\l(s)} j^2_{m} I_{N/2-1}''(j_{m})}\cos( x)+ o(s)\\
&=  \frac{1}{j_{m}} \sqrt{\mu_s} - s\frac{j_{N/2-1,1}^2 I_{N/2-1}(j_{m}) I_{N/2}(j_{N/2-1,1}) }{j^2_{m} I_{N/2-1}''(j_{m})\sqrt{j_{m}^2-j_{N/2-1,1}^2 }}\cos( x)+ o(s) \qquad \text{as $s \to 0$}.
\end{align*}
Finally, by (\ref{eq:Solutionfinal}), (\ref{eq:def-h-v-phi-s}), (\ref{eq:varphi-s-exp}) and (\ref{eq:varphi-s-der-exp}), 
\begin{align}
&w_s(\frac{t}{h(s)},x)= \ti u_s ( \t,x) = u_m(\t,x) + \varphi_{s}({\t},x) - g(|t|)h_{\vp_{s}}(x) \nonumber\\
&=U_m(|t|)  + s \Bigl( I_{N/2-1}(j_{N/2-1,1}|t| )  + \frac{ j_{N/2-1,1}^2 I_{N/2-1}(j_{m}) I_{N/2}(j_{N/2-1,1}) }{j_m^2 I_{N/2-1}''(j_{m})} g(|t|) \Bigr)\cos (x) + o(s) \nonumber\\ 
  &=U_m(|t|) + s \Bigl( I_{N/2-1}(j_{N/2-1,1}|t|)+ \frac{j_{N/2-1,1}^2 I_{N/2-1}(j_{m})  I_{N/2}(j_{N/2-1,1}) }{j_m^2 I_{N/2-1}''(j_{m})}|t|U_m'(|t|) \Bigr)\cos (x) + o(s),\nonumber
\end{align}
where $o(s) \to 0$ in $C^1$-sense on $\Omega_*$.
We thus have proved Theorem~\ref{Theo1-ND} with the constants
\begin{align*}
\mu_0& =\frac{j_m^2}{\lambda(0)}=\frac{j_m^2}{\lambda_m}= \frac{j_m^2}{j_m^2 - j_{N/2-1,1}^2},\qquad \quad   \kappa = \frac{1}{j_{m}}= \frac{1}{j_{N/2,m}},\\
\beta &=- \frac{j_{N/2-1,1}^2 I_{N/2-1}(j_{N/2,m}) I_{N/2}(j_{N/2-1,1}) }{j_{N/2,m}^2 I_{N/2-1}''(j_{N/2,m})\sqrt{j_{N/2,m}^2-j_{N/2-1,1}^2 }},\\
\gamma&=- \frac{j_{N/2-1,1}^2 I_{N/2-1}(j_{N/2,m})  I_{N/2}(j_{N/2-1,1}) }{j_{N/2,m}^2 I_{N/2-1}''(j_{N/2,m})}
\end{align*}
and with the function
$$
t \mapsto \phi_1(|t|)= I_{N/2-1}(j_{N/2-1,1}|t|)
$$
which is a Dirichlet eigenfunction of $-\Delta$ on $B_1$ associated with the first eigenvalue $j_{N/2-1,1}^2$.
\QED

\section{The Schiffer problem on $S^2$}
\label{sec:schiffer-problem-s2}
In this section, we study problem $(\textrm{N}_\mu)$ for a family of subdomains of $S^2$, and we prove Theorem~\ref{Theo1-ND-S2}.   
As before for the case $N=2$, we set
$$
\O_*:=  (-1,1) \times \R.
$$
Our aim is to find constants $\mu>0$ and nonconstant functions $h \in \cP^{2,\alpha}_{S^2}(\R)$ with the property that the overdetermined problem (\ref{h-Neu-over-S2}) admits a nontrivial solution, where $\cP^{2,\alpha}_{S^2}(\R)$ is defined in (\ref{eq:def-p-s-2}) and $\widetilde \Omega_{h}$ is defined in (\ref{eq:Pertsphere-1}).  It is convenient to pull back problem (\ref{h-Neu-over-S2}), for given $h \in \cP^{2,\alpha}_{S^2}(\R)$ to the fixed domain $\Omega_* \subset \R^2$ by means of the parametrization  
$$
\Upsilon_h: \Omega_* \to S^2, (t,x) \mapsto \left(\cos x \cos \bigl(h(x)t\bigr),\sin x \cos \bigl(h(x)t\bigr), \sin \bigl(h(x)t\bigr) \right).
$$
The metric in the coordinates  $(t,x)$ is then given by 
\begin{equation}\label{eq:me}
g_{h}=h^2(x) \textrm{d}t^2 + \Bigl(\cos^2 \bigl(h(x)t) + \bigl(t h'(x)\bigr)^2 \Bigr) d x^2 + t h(x)  h'(x)\textrm{d}t dx.
\end{equation}
Consequently, the Gram determinant and inverse metric tensor are given by
$|g_h(t,x)|= h^2(x) \cos^2\bigl(h(x)t\bigr)$ and
$$
(t,x) \mapsto \frac{1}{\Bigl(h^2(x) \cos^2\bigl(h(x)t\bigr)\Bigr)}
\Bigl[\Bigl(\cos^2 \bigl(h(x)t) + \bigl(t h'(x)\bigr)^2 \Bigr) \textrm{d}t^2 + h^2(x)  d x^2 - t h(x) h'(x)\textrm{d}t dx\Bigr].
$$
Thus the corresponding Laplace-Beltrami operator of $S^{2}$ writes in these coordinates as 
\begin{align*} 
\Delta_{g_h}u &= \Bigl[\frac{1}{\sqrt{|g|}}\partial_i \Bigl(\sqrt{|g|}g^{ij}\partial_j u \Bigr)\Bigr]\\
  &=\frac{1}{h(x) \cos \bigl(h(x)t\bigr)}\Bigl[\partial_t \Bigl( \Bigl(\frac{\cos \bigl(h(x)t\bigr)}{h(x)} + \frac{\bigl(t h'(x)\bigr)^2}{h(x) \cos \bigl(h(x)t\bigr)} \Bigr)\partial_tu \Bigl)\\
  &+ \partial_{x}  \Bigl(\frac{h(x)}{\cos \bigl(h(x)t\bigr)}\partial_x u\Bigl)
- \partial_t  \Bigl(\frac{t h'(x)}{\cos \bigl(h(x)t\bigr)}\partial_x u\Bigl) - \partial_x \Bigl(\frac{t h'(x)}{\cos \bigl(h(x)t\bigr)}\partial_t u\Bigl) \Bigr]\\
&=\frac{1}{h^{2}(x)}\partial_{tt}u-\frac{1}{h(x)} \tan \bigl(h(x)t\bigr) \partial_tu \nonumber\\
           &+\frac{1}{\cos^2 \bigl(h(x)t\bigr)}\Bigl( \partial_{xx}u  + \frac{h'(x)^2}{h^2(x)}t^2 \partial_{tt}u - 2 \frac{h'(x)}{h(x)}t    {\partial}_{t} \partial_{x}u + \Bigl(2 \frac{(h'(x))^2}{h^2(x)} -\frac{h''(x)}{h(x)}\Bigr) t \partial_tu\Bigr)  .
\end{align*}
Hence (\ref{h-Neu-over-S2}) is equivalent to the problem
\begin{equation}\label{h-Neu-over-S2-pull-ball}
  \left \{
    \begin{aligned}
\cL_\mu^h w &=  0 && \qquad \text{in $\Omega_{*}$,}\\
             w&=1 &&\qquad \text{on $\partial  \Omega_{*}$,}\\
              \partial_t w  &=0 &&\qquad \text{on $\partial \Omega_{*}$}
    \end{aligned}
       \right.
\end{equation}
with
\begin{equation}
  \label{eq:def-L-mu-h-sphere}
\cL_\mu^h = \Delta_{g_h} + \mu \id   .
\end{equation}
To deal with this problem, we introduce very similar function spaces as in Section~\ref{sec:functional-setting}. For fixed $\a\in (0,1)$ and $k \in \N \cup \{0\}$, we define  $C^{k,\a}_{p}(\overline \Omega_*)$ the subspace of  functions in $C^{k,\a}(\overline \Omega_*)$  which are even in each of  their  variables and periodic in their last variable.  We also define
$$
\cX_k:= \{ u \in C^{k,\a}_{p}(\overline \Omega_*)\::\:  \de_\t u \in C^{k,\alpha}(\overline \Omega_*) \},
$$
endowed with the norm $u \mapsto \|u\|_{k}:= \|u\|_{C^{k,\alpha}}+\|{\de}_\t u \|_{C^{k,\alpha}}$, and we consider the closed subspaces
$$
\cX_k^D:=\{u \in \cX_k \::\: \text{$u= 0$ on $\partial \Omega_*$}\},
$$
and
$$
\cX_k^{DN}:=\{u \in \cX_k \::\: \text{$u={\de}_\t u = 0$ on $\partial \Omega_*$}\},
$$
both also endowed with the norm $\|\cdot\|_k$. Moreover, we consider the Banach space
$$
\cY:= C^{1,\alpha}_{p}(\overline \Omega_*) + \cX_0^D \; \subset \; C^{0,\alpha}_{p}(\overline \Omega_*),
$$
endowed with the norm
$$
\|f\|_{\cY}:= \inf \Bigl \{\|f_1\|_{C^{1,\alpha}} + \|f_2\|_{0} \::\: f_1 \in C^{1,\alpha}_{p}(\overline \Omega_*), \; f_2 \in \cX_0^D,\; f= f_1 + f_2 \Bigr\}.
$$
Next, for $\l\in [0,\frac{\pi}{2})$,  we  let $\mu(\lambda)$ denote the first positive eigenvalue of the eigenvalue problem
  \begin{equation}
    \label{eq:scalar-eigenvalue}
    \left\{
      \begin{aligned}
        &-\partial_{tt} u+ \lambda \tan (\lambda t)  \partial_tu=-\frac{1}{\cos(t\l)}\de_t(\cos(t\l)\de_tu) = \mu u \quad \text{in $(0,1)$,}\\
        &\qquad u'(0)=u'(1)=0.     
    \end{aligned}
\right.
\end{equation}
It is well known that $\mu(\lambda)$ is a simple eigenvalue which is given variationally as
\be\label{eq:var-char-mu-Neum}
 \mu(\lambda)=\inf_{u\in H^1(0,1), \int_0^1 u(t)  \cos(t\l)dt=0 } \frac{ \int_{0}^{1} (u')^2\cos(t\l)  dt    }{   \int_{0}^{1} u^2\cos(t\l)  dt  }. 
\ee
Hence $\mu(\lambda)$ admits a unique eigenfunction $U_\lambda$ satisfying the normalization condition 
  $$
 U_\lambda(1)= 1, 
  $$
and $U_\lambda$ changes sign only once on $(0,1)$. Moreover, the maps
  $$
  (-1,1) \to \R,\quad \lambda \mapsto \mu(\lambda),\quad \qquad (-1,1) \to C^{2,\alpha}([0,1]),\quad  \l \mapsto  U_\l
  $$
  are of class $C^2$ (see e.g. \cite{Kong-Zettl}).
  
  For matters of convenience, in the following we identify, for $\lambda \in (0,\frac{\pi}{2})$, the function $U_\lambda$ with its unique extension on the whole intervall $[0,\frac{\pi}{2\l})$ given as the solution of (\ref{eq:scalar-eigenvalue}) with $U_\lambda(1)= 1$. We also define
    $$
    u_\lambda \in C^{\infty}\bigl((-\frac{\pi}{2\l},\frac{\pi}{2\l}) \times \R\bigr), \qquad u_\lambda(t,x) = U_\lambda(|t|).
    $$   
    and we identify $u_\lambda$ with its restriction to $\overline{\Omega_*}$.   For $\lambda \in (0,1)$ and
    $$
    h \in \cU_{S^2}:= \Bigl\{ \tilde h \in C^{2,\alpha}_{p}(\R) \::\: \tilde h > \frac{2}{\pi}-1 \Bigr\},  
$$
we now consider the operator
  \begin{equation}
  \label{eq:def-L-lambda-h-sphere}
  \begin{split}
  u \mapsto L_\lambda^hu &:= \lambda^2 \cL_{\mu(\lambda)/\lambda^2}^{\lambda/ (1+ h)}u =
\mu(\lambda) u  + (1 +h)^2\partial_{tt} u- \lambda(1+ h) \tan \frac{\lambda t}{(1+h)}  \partial_tu \\
  &+\frac{\lambda^2}{\cos^2 \frac{\lambda  t}{(1+h)}}\Bigl( \partial_{xx}u  + \frac{h'^2}{(1+h)^2} t^2 \partial_{tt}u   + 2 \frac{h'}{1+h}t    {\partial}_{t} \partial_{x}u +  \frac{h''}{1+h} t \partial_t u \Bigr)     .
  \end{split}
\end{equation}
Here, in the second equality, we used the identities
$$
\Bigl(\frac{1}{1+h}\Bigr)' = - \frac{h'}{(1+h)^2},\qquad \Bigl(\frac{1}{1+h}\Bigr)'' = 2 \frac{(h')^2}{(1+h)^3}-\frac{h''}{(1+h)^2}.
$$
Moreover, for $\lambda \in (0,1)$ we define
\begin{equation}\label{eq:MappinmgF-ND-S2}
  F_{\lambda}:   \cX_2^{DN} \times \cU_{S^2} \to \cY, \qquad F_{\lambda}(u, h) = L_\lambda^h (u+u_{\lambda}).
  \end{equation}
By construction, we then have the following:
\begin{equation}
  \label{eq:F-equivalence}
  \begin{split}
    &\textit{If $F_{\lambda}(u, h)=0$, then the function $\tilde u = u_\lambda + u$ solves the problem (\ref{h-Neu-over-S2-pull-ball})}\\
    &\textit{with $h$  replaced by $\frac{\lambda}{1+h}$ and $\mu = \mu(\lambda)/\lambda^2$.}
  \end{split}
\end{equation}
To see that indeed $F_\lambda$ is well-defined by (\ref{eq:MappinmgF-ND-S2}), we argue similarly as in the proof of Lemma~\ref{sec:functional-setting-1} by writing
$F_{\lambda} = F_{\lambda}^1 + F_{\lambda}^2$ with
$$
F_\lambda^1: \cX_2^{DN} \times \cU_{S^2} \to C^{1,\alpha}_{p}(\overline \Omega_*), \qquad \text{and}\qquad F_\lambda^2: \cX_2^{DN} \times \cU_{S^2} \to \cX_0^D
$$
given by 
\begin{align*}
F_\lambda^1(u,h)&=(1 +h)^2 \partial_{tt} u- \lambda (1+ h) \tan \frac{\l t}{ 1+h}  \partial_tu \\
  &+\frac{\lambda^2}{\cos^2 \frac{\lambda t}{(1+h)}}\Bigl(\frac{(h')^2}{(1+h)^2} t^2 \partial_{tt}u  + 2 \frac{h'}{1+h}t    {\partial}_{t} \partial_{x}u \Bigr),\\ 
F_\lambda^2(u,h)&=\mu(\lambda) u +\frac{\lambda^2}{\cos^2 \frac{\lambda t}{(1+h)}}\Bigl( \partial_{xx}u +  \frac{h''}{(1+h)} t \partial_t u\Bigr).
\end{align*}
Next, similarly as in Section~\ref{sec:functional-setting}, we wish to eliminate the variable $h$. For this we define the linear map
\begin{equation}
\label{def-M-S2}
M: \cX_2^D \to \cX_2^{DN} \times  C^{2,\alpha}_{p}(\R),\qquad  M u = (M_1 u,h_u)
\end{equation}
with
\begin{equation}\label{eq:DeffM-ND--2-S2}
h_u(x)= \frac{{\de}_t u (1,x)}{U_\l''(1)}, \qquad x \in \R
\end{equation}
and
\begin{equation}\label{eq:DeffM-ND--1-S2}
[M_1 u](t,x)= u(t,x) - t U_\l'(|t|) h_u(x), \qquad (t,x) \in \Omega_*.
\end{equation}
By construction, the linear map $M$ is well defined by (\ref{def-M-S2}). We then define 
\begin{equation}\label{eq:DeffGl-ND-S2}
G_\lambda: \cU  \to \cY,\qquad G_\lambda = F_\lambda \circ M
\end{equation}
with the open subset $\cU:= \left\{u \in \cX_2^D\::\: \, h_u  \in \cU_{S^2} \;\right \}$ of $\cX_2^D$. We then note the following:
\begin{equation}
  \label{eq:G-equivalence}
  \begin{split}
    &\textit{If $G_\lambda(u)= 0$ for some $u \in \cX_2^D$, $\lambda \in (0,1)$, then the function $M_1 u +u_\l$}\\
    &\textit{solves (\ref{h-Neu-over-S2-pull-ball}) with $\mu$ replaced by $\mu(\lambda)/\lambda^2$ and  $h$ replaced by $\frac{\lambda}{1+h_u}$.}
  \end{split}
\end{equation}
We also note that
\begin{equation}
  \label{eq:trivial-branch-S-2}
G_\lambda(0) = F_\lambda(0,0) \qquad \text{for all $\lambda \in (0,1)$.}
\end{equation}

\begin{Proposition}
  \label{sec:functional-setting-4-S2}
For every $\lambda \in (0,1)$, the map $G_\l : \cU \subset \cX_2^D \to \cY $ defined by  \eqref{eq:DeffGl-ND-S2} is of class $C^\infty$.  Moreover for all  $v\in \cX_2^D$ we have 
\begin{align}\label{eequivla-ND-S2}
D G_\lambda(0)v=  L_\l^0 v= \mu(\lambda) v  + \partial_{tt} v- \lambda \tan (\lambda t)  \partial_tv +  \frac{\lambda^2}{\cos^2 (\lambda t)} \partial_{xx}v .
\end{align}
\end{Proposition}

\begin{proof}
  Similarly as in the proof of Proposition~\ref{sec:functional-setting-4}, the $C^\infty$-regularity of $G_\lambda$ follows from the $C^\infty$-regularity of the map $F$. To derive (\ref{eequivla-ND-S2}), we recall that by construction we have
  $$
  \Bigl(\lambda^2 \Delta_{g_\lambda}+  \mu(\lambda)\Bigr)u_\l = -\Bigl( -\partial_{tt} u_\l + \lambda \tan (\lambda t)  \partial_t u_\l - \mu(\lambda) u_\l\Bigr)= 0
  $$
  on $(-\frac{\pi}{2\l},\frac{\pi}{2\l}) \times \R$. Consequently, the function 
  $$
  (t,x) \mapsto v_{\lambda,h}(t,x)= U_\lambda (\frac{|t|}{1+h(x)}) = u_\lambda(\frac{t}{1+h(x)}, x)
  $$
  satisfies
  \begin{align*}
  [L_{\lambda}^h v_{\lambda,h}](t,x) &=
                                       \lambda^2 [\cL_{\mu(\lambda)/\lambda^2}^{\lambda/ (1+ h)}v_{\lambda,h}](t,x) =
                                       \Bigl[\lambda^2 \Delta_{g_{\frac{\lambda}{1+ h}}}v_{\lambda,h} +\mu(\lambda)v_{\lambda,h}\Bigr](t,x)\\
&=\Bigl[\Bigl(\lambda^2 \Delta_{g_\lambda}+ \mu(\lambda)\Bigr)u_\l\Bigr](\frac{t}{1+h(x)},x) = 0 \qquad \text{for $(t,x) \in \Omega_*$.}
  \end{align*}
  Differentiating gives, for fixed $h \in C^{2,\alpha}_p(\R)$,
  \begin{equation}
    \label{eq:diff-gives}
0= \frac{d}{ds}\Bigl|_{s=0}  \Bigl [L_{\lambda}^{sh} v_{\lambda,sh}\Bigr] =
\Bigl(\frac{d}{ds}\Bigl|_{s=0} L_{\lambda}^{sh}\Bigr)u_\lambda + L_\lambda^0 w_{\lambda,h}
  \end{equation}
  with
  $$
  w_{\lambda,h}(t,x) = \Bigl(\frac{d}{ds}\Bigl|_{s=0} v_{\lambda,sh}\Bigr)(t,x) = - t U_\lambda' (|t|)h(x) 
  $$
  for $(t,x) \in \Omega_*$. We can thus follow the proof of Proposition~\ref{sec:functional-setting-4} and see that for $v\in \cX_2^D$ we have
  $$
  [M_1 v](t,x) = v(t,x)- t U_\lambda' (|t|)h_v(x) = v(t,x) +   w_{\lambda,h_v}(t,x)
  $$
  and therefore, by the chain rule and (\ref{eq:diff-gives}), 
\begin{align*}
  DG_\l(0)v&=\de_u F_\l(0,0)M_1 v+ \de_h F_\l(0,0)h_v\\
  &= \cL_{\mu(\lambda)/\lambda^2}^{\lambda}M_1 v + \Bigl(\frac{d}{ds}\Bigl|_{s=0} L^{sh_v}_\l\Bigr)u_\lambda \\
           &= L_{\lambda}^0 \Bigl(v+ w_{\lambda,h_v}\Bigr) + \Bigl(\frac{d}{ds}\Bigl|_{s=0} L^{ sh_v}_\l\Bigr)u_\lambda = L_{\lambda}^0 v,
\end{align*}
as claimed. \QED 
\end{proof}
Next we wish to show the following.

\begin{Proposition}
\label{fredholm-S2}  
For every $\lambda \in (0,1)$, the operator $L_\lambda^0 = DG_\lambda : \cX_2^D \to \cY$ is a Fredholm operator of index zero.
\end{Proposition}

\begin{proof}
It suffices to show that 
\begin{equation}
  \label{eq:suffices-iso-sphere}
  \begin{split}
    &\text{$L_\l^0-\mu(\l)\id  = \frac{1}{\cos(\l t)}\de_t \cos(\l t)\de_t +\frac{\lambda^2}{\cos^2(\l t)} \de_{xx} \;:\; \cX_2^D \to \cY$}\\
   &\text{is a topological isomorphism,}  
   \end{split}
\end{equation}
since the inclusion $\cX_2^D \to \cY$ is  a compact operator. An integration by parts argument as in the proof of Proposition~\ref{fredholm} shows that $L_\l^0-\mu(\l)\id$ is injective. Moreover, it is well known that $L_\l^0-\mu(\l)\id: C^{2,\a}_p(\ov{\O_*})\cap\cX_0^D \to C^{0,\a}_p(\ov{\O_*})$ is an isomorphism. This implies, in particular, that for every $f \in \cY \subset C^{0,\a}_p(\ov{\O_*})$ there exists $u \in C^{2,\a}_p(\ov{\O_*})\cap\cX_0^D$ with $\bigl(L_\l^0-\mu(\l)\id\bigr)u = f$. It thus remains to show that
\begin{equation}
  \label{eq:remains-to-show-sphere}
u \in \cX_2^D.  
\end{equation}
Since $u \in C^{2,\alpha}(\ov{\Omega_*})$, we have 
  $v:={\de}_t u \in C^{1,\alpha}_{p}(\ov{\Omega_*}).$  Moreover, it is easy to see, by differentiating the equation $L_\l^0u-\mu(\l)u=0$ in distributional sense, that $v$ satisfies  
   \be\label{eq:Dtuo-S2}
  L_\l^0 v= \mu(\l)v + \de_t f- [\lambda \de_t(\tan (\lambda t))] \partial_tv     + \Bigl[\de_t\Bigl(\frac{\lambda^2}{\cos^2 (\lambda t)} \Bigr)\Bigr] \partial_{xx}v   \qquad\textrm{ in $\O_*$}
  \ee
  in distributional sense.  Note  that $\de_t f \in C^{0,\alpha}(\ov{\Omega_*})$,  since $f \in \cY$.  Hence the RHS of \eqref{eq:Dtuo-S2} belongs to $C^{0,\alpha}(\overline{\O_*})$.   In addition,
\be \label{eq:detv}
{\de}_t v \Big|_{\de{\O_*}}=   \de_{tt} u\Big|_{\de{\O_*}}   =\Bigl(\pm \lambda \tan (\lambda) \de_t u + f\Bigr)\Big|_{\de{\O_*}} \in C^{1,\alpha}(\de{\O_*}) ,
\ee
 since $u \in C^{2,\alpha}(\ov{\Omega_*})$, $u\equiv 0$ on $\de\O_*$ and 
  $$
  f = f_1 + f_2 \in \cY\qquad \text{with $f_1 \in C^{1,\alpha}_{p}(\ov{\O_*})$ and $f_2 \in \cX_0^D$.}
  $$
Consequently, standard elliptic regularity, applied to the Neuman boundary value problem \eqref{eq:Dtuo-S2}-\eqref{eq:detv}  gives $v={\de}_t u \in C^{2,\alpha}(\overline{\O_*})$, as required. \QED 
\end{proof}

Next, for $\l\in [0,1)$ and $\ell\in \N\cup\{0\}$,  we denote by $\s_{\ell}(\l)$ the first eigenvalue of the eigenvalue problem
\begin{equation}
(E)_{\ell,\lambda} \quad 
\left\{\begin{aligned}
&-\frac{1}{\cos(\lambda t)}\de_t \cos(\lambda t)(\partial_{t} v) + \frac{(\ell \lambda)^2}{\cos^2 (\lambda t)}v -   \mu(\lambda) v =\s  v\quad \text{in $(0,1)$,}\\
&v'(0)=v(1)=0,
\end{aligned}
\right.
\label{eq:ell-eigenvalue-problem}
\end{equation}
  which is characterized variationally as
  \begin{align}
    \label{eq:s-ell-var-char}
    \s_{\ell}(\l) &= \inf_{v\in \cH(0,1)} \frac{ \int_{0}^{1} (v')^2\cos(t\l)  dt + (\ell \l)^2\int_{0}^{1} \frac{ v^2}{\cos(t\l)}  dt }{   \int_{0}^{1} v^2\cos(t\l)  dt  }\; - \; \mu(\l)  ,
\end{align}   
with $\cH(0,1):= \{u \in H^1(0,1)\::\: \text{$u(1)=0$ in trace sense}\}$. It is well known that $\s_{\ell}(\l)$ is simple and that eigenfunctions are smooth on $[0,1]$, so there exists a unique eigenfunction $V_\lambda \in C^2([0,1])$ satisfying $V_\lambda'(0)=0, V_\lambda(0)=1$ and $V_\lambda(1)=0$. Moreover,
\begin{equation}
  \label{eq:positivity-v-lambda}
\text{$V_\lambda$ is positive on $(0,1)$.}   
\end{equation}
We also note that the functions $\l\mapsto \s_{\ell}(\l) $ and $\l\mapsto V_\l \in C^2([0,1])$ are differentiable on $[0,1)$.
We need the following further information on the function $\l \mapsto \s_{\ell}(\l)$.

  \begin{Lemma}\label{eq:lem-S2-eig-1}
    \begin{itemize}
    \item[(i)]  We have $\s_{\ell}(0)= -\frac{3}{4}\pi^2$ and
      \begin{equation}
        \label{eq-first-est}
       \s_{\ell}(\l) \ge \s_0(\l) + \ell^2 \lambda^2  \qquad \text{for $\l \in (0,1)$.}  
      \end{equation}
    \item[(ii)] We have  
      \begin{equation}
      \sigma_{\ell}'(\l) \ge (2 \ell^2 -C)\l \qquad \quad \text{for $\lambda \in (0,1)$}       \label{eq:formular-der-s-ell}
    \end{equation}
    with a constant $C>0$ independent of $\l$ and $\ell$.
    \item[(iii)] For all $\lambda_0\in (0,1)$ there exists $\ell_0= \ell_0(\l_0)$ with the property that for all $\ell\geq\ell_0$, there exists a unique $\l_* \in (0,\l_0)$ such that  $\s_{\ell}(\l_*) =0$ and $\s_{\ell}'(\l_*) >0$.
    \end{itemize}
  \end{Lemma}
  \begin{proof}
    (i) It is well known that the first eigenvalues of (\ref{eq:scalar-eigenvalue}) and $(E)_{\ell,\lambda}$ for $\lambda = 0$ are given by $\mu(0)= \pi^2$ and $\s_\ell(0) = \frac{\pi^2}{4} - \mu(0) =  -\frac{3}{4}\pi^2$ with eigenfunctions $U_0(t)= \cos (\pi t)$ and $V_0(t)= \cos(\frac{\pi}{2}t)$.\\
    To see (\ref{eq-first-est}), we note that by (\ref{eq:s-ell-var-char}) we have   
    $$
    \s_{\ell}(\l) \ge \inf_{v\in \cH(0,1)} \frac{ \int_{0}^{1} (v')^2\cos(t\l)  dt}{   \int_{0}^{1} v^2\cos(t\l)  dt  } + (\ell \l)^2 \inf_{v\in \cH(0,1)} \frac{\int_{0}^{1} \frac{ v^2}{\cos(t\l)}  dt }{   \int_{0}^{1} v^2\cos(t\l)  dt  }
-\mu(\l)  \ge \s_0(\l)+ (\ell \l)^2
$$   
since $\frac{1}{\cos(t\l)} \ge 1 \ge \cos(t\l)$ for $\lambda,t \in [0,1]$.\\ 
(ii) We  recall that 
\begin{equation}
  \label{eq:v_l-sol}
\displaystyle-\frac{1}{\cos(\l t)}\de_t(\cos(\l t)\de_t V_\l )  + \frac{(\ell \lambda)^2}{\cos^2 (\lambda t)} V_\l- \mu(\l) V_\l=\s_{\ell}(\l)  V_\l \qquad\textrm{ on (0,1)}
\end{equation}
and $V_\l'(0)=V_\l(1)=0$. Differentiating with respect to $\l$  and setting $w_\l=\de_\l V_\l $,  we get 
\begin{align*}
\s_{\ell}'(\l) V_\l &= -\partial_{tt} w_\l + \lambda \tan (\lambda t)  \partial_t w_\l + \frac{(\ell \lambda)^2}{\cos^2 (\lambda t)} w_\l - \mu(\l)w- \s_{\ell}(\l) w_\l \\
  &+ \de_\l( \lambda \tan (\lambda t))  \partial_t V_\l - \mu'(\l) V_\l+ \de_\l\left( \frac{(\ell \lambda)^2}{\cos^2 (\lambda t)} \right)V_\l.
\end{align*}
We multiply this by $V_\l \cos(t\l)$,  integrate over $(0,1)$ and use  \eqref{eq:v_l-sol} to obtain
\begin{align}
  &\s_{\ell}'(\l)  \int_{0}^{1} V_\l^2  \cos(t\l) dt = - \frac{1}{2}\int_{0}^{1} \de_t \bigl[ \cos(t\l) \de_\l( \lambda \tan (\lambda t))\bigr]  V_\l^2  dt  \nonumber\\
  &-\mu'(\l)   \int_{0}^{1} V_\l^2 \cos(t\l)  dt +\int_{0}^{1}   \de_\l\left( \frac{(\ell \lambda)^2}{\cos^2 (\lambda t)} \right)V_\l^2  \cos(t\l) dt,\label{final-extra}
\end{align}
where
\begin{align*}
  &-\de_t \bigl[ \cos(t\l) \de_\l( \lambda \tan (\lambda t))\bigr]= -\de_t \bigl[ \cos(t\l)\bigl( \tan (\lambda t)+ \frac{\lambda t}{\cos^2(\lambda t)}\bigr)\bigr]\\ 
  &= -\de_t \bigl[ \sin (\lambda t)+ \frac{\lambda t}{\cos(\lambda t)}\bigr] = \lambda \rho(\lambda t)
\end{align*}
with
$$
\rho(\tau) := -\partial_\tau \bigl[ \sin \tau+ \frac{\tau}{\cos \tau }\bigr] =-\Bigl( \cos \tau + \frac{1 + \tau \tan \tau}{\cos \tau}\Bigr). 
$$
A similar (but easier) argument for the function $\lambda \mapsto \mu(\l)$  using   \eqref{eq:scalar-eigenvalue}
gives 
\begin{align}
\label{eq:der-lam1}
& \mu'(\l) \int_{0}^{1} U_\l^2  \cos(t\l) dt = \frac{\l}{2}\int_{0}^{1}  \rho(\l t) U_\l^2  dt.
\end{align}
We note that the function $\tau \mapsto \frac{\rho(\tau)}{\cos \tau}$ is decreasing on $[0, \frac{\pi}{2})$ since, by direct compuation,
$$
\partial_\tau \frac{\rho(\tau)}{\cos \tau} =-\partial_\tau \Bigl(1 +\frac{1 +\tau \tan \tau}{\cos^2 \tau}\Bigr) =-
\frac{1}{\cos^2 \tau} \bigl(3 \tan \tau +2\tau  \tan ^2 \tau + \frac{\tau}{  \cos^2 \tau }\bigr)<0
$$
for $\tau \in (0,\frac{\pi}{2})$. Therefore we have 
\begin{align}\label{eq: boudrho}
\frac{\rho(1)}{\cos 1} \leq  \frac{\rho(\lambda t)}{ \cos ( \lambda t)}\leq  \frac{\rho(0)}{\cos 0} = -2 \qquad \text{for every $\lambda, t \in [0, 1]$.}
 \end{align} 
Using \eqref{eq: boudrho} in \eqref{final-extra}  and (\ref{eq:der-lam1}), we see that $\mu'(\l) \le  - 2 \lambda$ and 
\begin{align*}
  \s_{\ell}'(\l) &=  \frac{\l}{2}\frac{\int_{0}^{1} \rho (\l t) V_\l^2  dt }{\int_{0}^{1} V_\l^2  \cos(t\l) dt}  -\mu'(\l) + \ell^2 \frac{\int_{0}^{1}   \de_\l \left( \frac{\lambda^2}{\cos^2 (\lambda t)} \right)V_\l^2  \cos(t\l) dt}{\int_{0}^{1} V_\l^2  \cos(t\l) dt}\\
  &\ge \frac{\lambda}{2}\Bigl(\frac{\rho(1)}{\cos 1} +2 \Bigr) +  \ell^2\, \frac{\int_{0}^{1} \de_\l\left( \frac{\lambda^2}{\cos^2 (\lambda t)} \right)V_\l^2  \cos(t\l) dt}{\int_{0}^{1} V_\l^2  \cos(t\l) dt}\\
&\ge \Bigl(2 \ell^2 -C\Bigr)\lambda \qquad \text{for $\lambda \in [0,1]$} 
\end{align*}
with a constant $C>0$ independent of $\l$ and $\ell$. Here we used in the last step that 
\begin{align*}
 \de_\l\left( \frac{  \lambda^2}{\cos^2 (\lambda t)} \right)= 2\frac{\l \cos^2 (\lambda t) +t\l^2 \sin (t\l)\cos(t\l)}{\cos^4 (\lambda t)} \geq  2\l\cos^{-2} (\lambda t) \ge 2 \l  
\end{align*}
for $t, \l \in [0,1]$. Hence (\ref{eq:formular-der-s-ell}) is proved.\\
(iii) Let $\lambda_0\in (0,1)$ be given. Choosing $\ell_0= \ell_0(\l_0) \in \N$ large enough, we can guarantee that $2 \ell_0^2 -C>0$ and
$\s_0(\l_0) + \ell^2_0 \lambda_0^2 >0$. Therefore from $(i)$,  (\ref{eq-first-est}),~(\ref{eq:formular-der-s-ell}) and the intermediate value theorem, it then follows that for all $\ell\geq \ell_0$ there exists a unique $\l_* \in (0,\l_0)$ such that  $\s_{\ell}(\l_*) =0$ and $\s_{\ell}'(\l_*) >0$.  \QED 
\end{proof}

As a consequence,  we may now deduce the following key proposition. 
\begin{Proposition}\label{propCR-ND-S2}
For a given $\lambda_0 \in (0,1)$,  let  $\ell\geq \ell_0=\ell_0(\lambda_0)$ and $\l_*=\l_*(\ell)\in (0,\l_0)$, where $\ell_0$ and $\l_*$ are given by Lemma \ref{eq:lem-S2-eig-1}. Then we have the following properties.
\begin{itemize}
\item[(i)] The kernel $N(L_{\l_*}^0)$ of the operator
$$
L_{\l_*}^0: \cX^D_2  \to \cY, \qquad L_{\l_*}^0  =DG_{\l_*}(0),
$$
is one-dimensional and spanned by the function
\begin{equation}\label{eq:kernel-ND-S2}
 v_{*}(t,x)=V_{\l_*}(|t|)\cos(\ell x).
\end{equation}
\item[(ii)] The image $R(L_{\l_*}^0)$ of $L_{\l_*}^0$ is given by 
$$
  R(L_{\l_*}^0 )= \left \lbrace  w\in \cY: \int_{\O_* } v_{*}(t,x)w(t,x) \cos(t\l_*)\,dxdt=0\right\rbrace.
$$
\item[(iii)] We have the transversality property 
\begin{equation*}
\partial_\lambda \Bigl|_{\lambda=\l_{*} }L_\lambda^0 v_{*} \not  \in \; R(L_{\l_*}^0).\\
\end{equation*}
\end{itemize}
\end{Proposition}

\begin{proof}
(i)   Let $w \in \cX_2^D $ be such that $L_{\l_*}^0w=0$ in $ \O_*$. By standard elliptic regularity theory, we have $w \in C^\infty(\ov{\O_*})$, and we can expand $w$ as a uniformly convergent Fourier series in the $x$-variable of the form $w(t,x)=\sum \limits_{k=0}^\infty w_k(|t|)\cos(kx)$, where, for every $k \ge 0$, the coefficient function $w_k(t):= \frac{1}{\sqrt{2\pi}}\int_{0}^{2\pi} w(t,x)\cos (k x)\,dx$ is an eigenfunction of the eigenvalue problem $(E)_{k,\l_*}$ (see (\ref{eq:ell-eigenvalue-problem})) 
corresponding to the eigenvalue $\sigma = 0$. Considering $w= w_k$ and $\s = 0$ in $(E)_{k,\l_*}$, multiplying the equation by $V_{\l_*}(t)\cos(t\l_*)$ and integrating by parts on $(0,1)$ gives
\begin{equation}
  \label{eq:ell-k-eq}
((\ell \l_*)^2-(k \l_*)^2)\int_{0}^{1} \frac{ w_k V_{\l_*}}{\cos(t\l_*)}  dt =0.
\end{equation}
If $k\not=\ell$ and $w_k \not =0$, it follows from (\ref{eq:positivity-v-lambda}) and (\ref{eq:ell-k-eq}) that $w_k$ changes sign in $(0,1)$. Therefore $0$ cannot be the first eigenvalue of the eigenvalue problem $(E)_{k,\l_*}$. Letting $\sigma_{2,k}(\lambda_*)>\sigma_{k}(\lambda_*)$ denote the second eigenvalue of $(E)_{k,\l_*}$, we thus deduce that
$$
0 \ge \sigma_{2,k}(\lambda_*) \ge \sigma_{2,0}(\lambda_*),
$$
where the second inequality easily follows from the nonnegativity of the function $t \mapsto \frac{(k \lambda_*)^2}{\cos^2 (\lambda_* t)}$ and the Courant-Fischer min-max principle characterizations of $\sigma_{2,k}(\lambda_*)$ and $\sigma_{2,0}(\lambda_*)$. The latter also shows, together with the variational characterization of $\mu(\lambda_*)$, that 
\begin{align*}
  \sigma_{2,0}(\lambda_*) &= \inf_{u\in \cH(0,1), \int_0^1 u(t)  \cos(t\l_*)dt=0 } \frac{ \int_{0}^{1} (u')^2\cos(t\l_*)  dt    }{   \int_{0}^{1} u^2\cos(t\l_*)  dt  }-\mu(\lambda_*)\\
 &\ge \inf_{u\in H^1(0,1), \int_0^1 u(t)  \cos(t\l_*)dt=0 } \frac{ \int_{0}^{1} (u')^2\cos(t\l_*)  dt    }{   \int_{0}^{1} u^2\cos(t\l_*)  dt  }-\mu(\lambda_*)=0. 
\end{align*}
Hence equality must hold in the last inequality, which then shows that the latter infimum is attained by a nonzero function $u\in \cH(0,1)$ which then solves (\ref{eq:scalar-eigenvalue}) with $\mu = \mu(\lambda_*)$ and the additional boundary condition $u(1)=0$.
This is impossible, and thus we conclude that $w_k \equiv 0$ for $k \not = \ell$.

This shows that the kernel $N(L_{\l_*}^0)$ is contained in the span of the function $v_*$ defined in \eqref{eq:kernel-ND-S2}. On the other hand, since $\sigma_\ell(\lambda_*)=0$ by Lemma \ref{eq:lem-S2-eig-1}(iii),  we also have $L_{\l_*}^0 v_* = 0$, and hence the claim follows.\\
(ii) Since $L_{\l_*}^0$ is a Fredholm operator of index zero by Proposition~\ref{fredholm-S2}, it suffices to show that 
$$
  R(L_{\l_*}^0 ) \subset \left \lbrace  w\in \cY: \int_{\O_* } v_{*}(t,x)w(t,x) \cos(t\l_*)\,dxdt=0\right\rbrace.
$$
This follows easily since for every $w = L_{\l_*}^0 u \in   R(L_{\l_*}^0 )$ we have, by integration by parts,
$$
\int_{\O_* } v_{*}(t,x)w(t,x) \cos(t\l_*)\,dxdt=\int_{\O_* } [L_{\l_*}^0 v_{*}](t,x)u(t,x) \cos(t\l_*)\,dxdt = 0.
$$
(iii) In the following, for $\lambda \in (0,1)$, we let $w_\lambda \in C^2(\overline{\Omega_*})$ be defined by
$$
w_\lambda (t,x) \mapsto V_\l(|t|)\cos (\ell x),
$$
so that $v_*= w_{\lambda_*}$ and $ L_\l^0 w_\l = \sigma_{\ell}(\lambda)w_\l$ for $\lambda \in (0,1)$ by definition of $V_\l$. Differentiating this identity
at $\lambda = \lambda_*$ 
and setting $w_*:= \frac{d}{d\lambda}|_{\l=\l_*} w_\lambda$, we obtain 
$$
\Bigl(\frac{d}{d\lambda}|_{\l=\l_*} L_\l^0\Bigr)v_{*} + L_{\l_*}^0 w_{*} = \sigma_{\ell}'(\lambda_*)v_{*} +
\sigma_{\ell}(\lambda_*)w_{*} = \sigma_{\ell}'(\lambda_*) v_{*},
$$
where $\sigma_{\ell}'(\lambda_*)>0$ by Lemma~\ref{eq:lem-S2-eig-1}. So $\Bigl(\frac{d}{d\lambda}|_{\l=\l_*} L_\l^0\Bigr)v_{*} \in   R(L_{\l_*}^0 )$ would imply that $v_* \in R(L_{\l_*}^0 )$, which is impossible by (ii). Hence (iii) holds.
\QED
\end{proof}

\subsection*{Proof of Theorem~\ref{Theo1-ND-S2} (completed)}

As in the case of Theorem~\ref{Theo1-ND}, the proof is obtained by an application
of the Crandall-Rabinowitz Bifurcation theorem. Let $\lambda_0 \in (0,1)$ be given, and choose $\ell_0 = \ell_0(\lambda_0)$ as in Lemma~\ref{eq:lem-S2-eig-1}. Moreover, for $\ell \ge \ell_0$, we choose $\l_* \in (0,\lambda_0)$ as in Lemma~\ref{eq:lem-S2-eig-1}, which implies that the claims of Proposition~\ref{propCR-ND-S2} hold with $L_\lambda^0 = D G_\lambda(0): \cX_2^D \to Y$. Consequently, by the Crandall-Rabinowitz Theorem (see \cite[Theorem 1.7]{M.CR}) applied to the map
$$
  (-\l_*,1-\l_*) \times \calU \to Y,\qquad (\lambda,u) \mapsto  G_{\lambda+\l_*}(u), 
$$
there  exist  ${\e_0}>0$ and a smooth curve
$$
(-{\e_0},{\e_0}) \to   (0,+\infty) \times \cX_2^D ,\qquad r \mapsto (\lambda(r) ,\varphi_{r})
$$
with $G_{\l(r)}(\varphi_{r}) =0$, $\lambda(0)= \l_*$ and 
\begin{equation}
  \label{eq:expansion-phi-s-s-2}
\varphi_{r} = r v_{*}+o(r) \qquad \text{in $\cX_2^D\quad $ as $r \to 0$,}
\end{equation}
where $v_{*}$ is given in (\ref{eq:kernel-ND-S2}). By~(\ref{eq:G-equivalence}), this means that problem (\ref{h-Neu-over-S2}) with
$$
\mu= \tilde \mu_r: = \mu(\lambda(r))/\lambda(r)^2, \qquad \quad h= \tilde h_r:= \frac{\lambda(r)}{1+h_{\varphi_r}}
$$
admits, for every $r \in (-\eps_0,\eps_0)$, a nontrivial solution. Moreover, by (\ref{eq:expansion-phi-s-s-2}) and the definition of $v_*$ we have
\begin{align*}
  \tilde h_r(x)= \frac{\lambda(r)}{ 1 + \frac{{\de}_t \varphi_r(1,x)}{U_{\l(r)}''(1) }} &= \lambda(r)\Bigl(1 -  r \frac{{\de}_t v_{*}(1,x)}{U_{\l(r)}''(1)} + o(r)\Bigr)\\
        &=\lambda(r) -  r \frac{\lambda_*  V_{\lambda_*}'(1)}{ U_{\l_*}''(1)}\cos(\ell x) + o(r). 
\end{align*}
We note that $U_{\l_*}''(1) \not = 0$ and $V_{\lambda_*}'(1) \not = 0$, since these functions are nontrivial solutions of linear
second order ODEs, and they satisfy $U_{\l_*}'(1)= 0$ and $V_{\lambda_*}(1)= 0$. Hence we may set $\eps:= \bigl|\frac{V_{\lambda_*}'(1)}{U_{\l_*}''(1)}\bigr|\lambda_* \eps_0>0$ and consider the reparametrisation
$$
(-{\e},{\e}) \to (-{\e_0},{\e_0}),\qquad s \mapsto r(s):= -\frac{s U_{\l_*}''(1) }{\lambda_*   V_{\lambda_*}'(1)}
$$
to define the curve
$$
(-{\e},{\e}) \to   (0,+\infty) \times (0,+\infty) \times \cX_2^D ,\qquad s \mapsto (\mu_s, \xi_s, h_s)
$$
with 
$$
\mu_s := \tilde \mu_{r(s)}, \qquad \xi_s:= \lambda(r(s)) \qquad \text{and}\qquad h_s := \tilde h_{r(s)}.
$$
Recalling that $\xi_0= \lambda(r(0)) = \lambda(0)= \lambda_* \in (0,\lambda_0)$, we may make $\eps>0$ smaller if necessary to
guarantee that $\xi_s \in (0,\lambda_0)$ for $s \in (-{\e},{\e})$. Consequently, the curve $s \mapsto (\mu_s, \xi_s, h_s)$ has the properties asserted in Theorem~\ref{Theo1-ND-S2}, and the proof is thus finished.
\QED

\end{document}